\documentclass[10pt]{article}
\usepackage[T1]{fontenc}
\usepackage{array}
\usepackage{amsmath,amsfonts,amsthm,mathrsfs,amssymb,bm}
\usepackage{cite}
\usepackage{graphicx,float}
\usepackage{subfigure}
\usepackage{placeins}
\usepackage{color}
\usepackage{indentfirst}
\usepackage{cases}
\usepackage{listings}
\usepackage[bookmarks=true]{hyperref}
\numberwithin{equation}{section}
\newcommand{\R}{{\mathbb R}}

\newcommand{\be}{\begin{eqnarray}}
\newcommand{\ben}{\begin{eqnarray*}}
\newcommand{\en}{\end{eqnarray}}
\newcommand{\enn}{\end{eqnarray*}}

\newcommand{\pa}{\partial}

\newtheorem{theorem}{Theorem}[section]
\newtheorem{lemma}[theorem]{Lemma}
\newtheorem{corollary}[theorem]{Corollary}

\newtheorem{remark}[theorem]{Remark}
\newtheorem{proposition}[theorem]{Proposition}

\allowdisplaybreaks[4]
\definecolor{rot}{rgb}{0.000,0.000,0.000}
\definecolor{rot1}{rgb}{0.000,0.000,0.000}

\newcommand{\jp}[1]{{\color{black}#1}}
\newcommand{\tao}[1]{{\color{black}#1}}

\begin{document}
\renewcommand{\theequation}{\arabic{section}.\arabic{equation}}
\begin{titlepage}
\title{Regularized boundary integral equation methods for open-arc scattering problems in thermoelasticity}
\author{Yixuan X. Kong\thanks{School of Mathematical Science, University of Chinese Academy of Sciences, China; State Key Laboratory of Mathematical Sciences and Institute of Computational Mathematics and Scientific/Engineering Computing, Academy of Mathematics and System Sciences, Chinese Academy of Sciences, Beijing, 100190, China. Email: {\tt kongxiangyixuan24@mails.ucas.ac.cn}}\;, 
Jos\'{e} Pinto\thanks{Facultad de Ingenier\'{i}a y Ciencias, Universidad Adolfo Ib\'a\~{n}ez, Santiago, Chile. Email: {\tt jose.pinto@uai.cl}}\;,
Tao Yin\thanks{State Key Laboratory of Mathematical Sciences and Institute of Computational Mathematics and Scientific/Engineering Computing, Academy of Mathematics and Systems Science, Chinese Academy of Sciences, Beijing 100190, China. Email: {\tt yintao@lsec.cc.ac.cn}}}

\end{titlepage}
\maketitle

\begin{abstract}
This paper devotes to developing novel boundary integral equation (BIE) solvers for the problem of thermoelastic scattering by open-arcs with four different boundary conditions in two dimensions. The proposed methodology is inspired by the Calder\'on formulas, whose eigenvalues are shown to accumulate at particular points depending only on Lam\'e parameters, satisfied by the thermoelastic boundary integral operators (BIOs) on both closed- and open-surfaces. Regularized BIEs in terms of weighted BIOs on open-arc that explicitly exhibits the edge singularity behavior, depending on the types of boundary conditions, of the unknown potentials are constructed to effectively reduce the required iteration number to solve the corresponding discretized linear systems. We implement the new formulations utilizing regularizations of singular integrals, which reduces the strongly- and hyper-singular integrals into weakly-singular integrals. Combined with spectrally accurate quadrature rules, numerical examples are presented to illustrate the accuracy and efficiency of the proposed solvers.

{\bf Keywords:} Thermoelastic scattering, boundary integral equation, edge singularity, spectrum
\end{abstract}

\section{Introduction}
\label{sec:1}

Developing accurate and efficient numerical solvers for wave scattering problems \cite{N01,CK13,HW08,L09,K13} in acoustics, elastodynamics, and electromagnetics remains challenging due to the highly oscillatory nature of the associated time-harmonic fields. Compared to volumetric discretization methods \cite{CB89,TSC06,WLXY}, such as the finite difference method and the finite element method, the boundary integral equation (BIE) method \cite{C00,HW08,BL12,LB15,BXY21,XY22,BL13,YHX17,ZXY21,BXY17,BXY19,CKM00,K95,BET12,L09,K13} offers two key advantages. First, it only requires the discretization of the boundary, reducing the problem's dimensionality. Second, the resulting integral representation naturally satisfies the radiation condition. Furthermore, the method can be combined with high-order discretization strategies and acceleration techniques \cite{BET12,CKM00,K95,L09} to achieve highly accurate numerical solutions, even at high frequencies.

In this work, we focus on the application of the BIE method for the thermoelastic problem on open-arcs, a problem of great importance in various engineering fields such as non-destructive testing, geophysics, and energy production (see \cite{E12,LB11,SIS16} for some applications). In contrast to closed boundary problems, where the BIE method is well understood \cite{DK88,CD98,N01,HW08,YHX17,ZXY21,BXY17,BXY19}, there are specific challenges associated with open-arcs. On one hand, the Calderón identities do not hold, making it impossible to obtain a second-kind formulation using the same approach as in the closed boundary case. Therefore, different techniques are needed to construct an appropriate preconditioner. On the other hand, the unknown density functions in the BIEs exhibit singular behavior near the edges of the open-arc, depending on the boundary conditions. This singular behavior poses considerable challenges to both the theoretical analysis and the numerical discretization of the BIOs. 

Recently, there have been numerous studies on open-arc and open-surface scattering problems aimed at establishing proper second-kind formulations. For instance, such formulations have been developed for elastic and acoustic waves in 2D \cite{BL12,LB15,BXY21,XY22} and for acoustic waves in 3D \cite{BL13}. The primary approach extends the results for closed boundaries by considering compositions of appropriately modified versions of the acoustic/elastic single-layer (denoted by $S^w$) and hyper-singular integral operators (denoted by $N^w$). The eigenvalues of the operator $N^w \circ S^w$
  were numerically computed and shown to remain bounded away from zero and infinity in \cite{BL12,BXY21}, with rigorous proofs provided in \cite{LB15,XY22}. Additionally, the referenced articles present numerical examples that demonstrate the effectiveness of this preconditioning strategy in reducing the number of iterations needed to solve the discrete systems.
As for spectral analysis of open-surface Calder\'on identities for other scattering problems, it still remains open.

We consider the elastic scattering problem by open arcs, where the influence of the temperature field cannot be neglected. In this situation, the Biot model can be employed to characterize the coupling effect between the elastic displacement field and the temperature field \cite{B56,B64,CD98,C00}. The mathematical theory of thermoelastic wave scattering by closed surfaces has been well developed in \cite{K79}, where four types of boundary conditions were considered. Furthermore, in \cite{C00}, the well-posedness of the BIEs for thermoelastic wave scattering by open surfaces, as well as the description of edge singularities for the four types of boundary conditions, were established. In this work, we focus on constructing second-kind formulations for thermoelastic wave scattering by open arcs.

Unlike other scattering problems, an additional challenge arises in constructing Calderón formulas on closed boundaries, as the double-layer operator 
$K$ and its adjoint 
$K^*$  are no longer compact. By analyzing the detailed structure of these operators, we show in Section \ref{sec:3.1} that the results from \cite{AJKKY18} can be used to establish the polynomial compactness of the double-layer operator and its adjoint. This, in turn, enables us to recover the second-kind formulation for the four types of boundary conditions on closed surfaces, as presented in Theorem~\ref{V12}.

Using the second-kind formulations for the closed boundary case and previous works \cite{BL12,LB15,BXY21,XY22} as inspiration, we construct new integral formulations by considering different compositions of weighted BIOs, where the exact combinations depend on the underlying boundary conditions.
Numerical examples confirm that the eigenvalues of the composite operators accumulate precisely at certain nonzero constants determined solely by the elastic medium parameters, suggesting that they are Fredholm operators of the second kind, even in high-frequency scenarios. We also present the detailed spectrum analysis for this regularized formulations. 
As a result, we observe that the new formulations require fewer GMRES iterations to achieve a given residual tolerance. The overall computing time, including the discretization of the BIOs, is also significantly reduced, except in the case of the pure Dirichlet problem, where the additional cost of assembling the new formulation outweighs the gains in solving the linear system.

Another unavoidable difficulty arising from the BIE method's implementation is the evaluation of strongly- and hyper-singular integrals defined in the sense of Cauchy principle value or Hadamard finite part, respectively. This paper adopts a method of establishing regularized formulations \cite{BXY21,YHX17,ZXY21}. 
By utilizing the techniques of G\"unter derivative and integration-by-parts, strongly- and hyper-singular integrals  on the closed-surface are transformed into a series of combination of weakly-singular integrals and tangential derivatives. 
For the open-arc case, the same results holds since the singular terms involved in the integration-by-parts calculation vanish at the end points. Furthermore, the resulted weakly-singular integrals can be discretized according to the Nystr\"om method \cite{CK13}, while the tangential derivatives can be calculated quickly using FFT. The proposed scheme offers the advantages of a simpler form and higher efficiency \cite{CKM00,K95}, in comparison with some alternative numerical treatments.   

The rest of this paper is organized as follows. Section~\ref{sec:2} introduces the thermoelastic scattering problem on the open-arc and derives four corresponding categories of BIEs. Section~\ref{sec:3.1} analyzes the spectral properties of the thermoelastic BIOs on the closed-surface, in which the appropriately composite operators are proven to be the second-kind Fredholm operators. Weighted BIOs and new regularized BIEs on the open-arc are presented in  Section~\ref{sec:3.2}, with their corresponding spectrum analysis in Section~\ref{sec:3.3}. Regularized representation formulas for the hyper-singular BIOs $V_i^w,i=2,3,4$ are then established in Section~\ref{sec:3.4}. Finally, in Section~\ref{sec:4.1} we present our numerical discretization strategy with high-order quadrature rules, and in  Section~\ref{sec:4.2} we show  ample numerical examples to demonstrate the high accuracy and efficiency of our algorithm.

\section{Preliminaries}
\label{sec:2}
\subsection{The thermoelastic scattering problem}
\label{sec:2.1}

Let $\Gamma$ denote a smooth open-arc in the plane $\mathbb{R}^2$. The
complement of $\Gamma$ in $\mathbb{R}^2$ is occupied by the medium of propagation, a linear isotropic and homogeneous thermoelastic medium of Biot type, characterized by the Lam\'e
constants $\lambda, \mu$ (satisfying $\mu>0$, $\lambda+\mu>0$),
the constant mass density $\rho>0$, the coefficient of thermal
diffusity $\kappa$ and the coupling constants $\gamma$ and
$\eta$ are given by
\ben
\gamma=(\lambda+\frac{2}{3}\mu)\,\alpha,\quad\eta=\frac{T_0\gamma}{\lambda_0},
\enn
respectively, where $\alpha$ is the volumetric thermal expansion coefficient,
$T_0$ is a reference value of the absolute temperature and $\lambda_0$
is the coefficient of thermal conductivity. Denote by $\epsilon$ :=$\gamma\eta\kappa/(\lambda+2\mu) $
the dimensionless thermoelastic coupling constant which assumes "small" positive values
for most thermoelastic media and $q=i\omega/\kappa$.
Suppressing the time-harmonic dependence
$e^{-i\omega t}$ in which $\omega>0$ is the frequency, the
displacement field $\bm u(\bm x)=(u_1(x_1,x_2), \,u_2(x_1,x_2))^{\top}$ and the temperature variation field $p(\bm x)=p(x_1,x_2)$ can be modeled by the following Biot
system of linearized thermoelasticity
\begin{numcases}{}
\label{navier}
\Delta^*\bm u+\rho\omega^2\bm u-\gamma\nabla p=0 \quad\mbox{in}\quad\R^2\backslash\Gamma,  \\
\label{navier2}
\Delta p+qp+i\omega\eta\nabla\cdot \bm u=0 \quad\mbox{in}\quad\R^2\backslash\Gamma.
\end{numcases}
Here $\Delta^*$ is the Lam\'e operator defined by
\ben
\Delta^* = \mu\,\mbox{div}\,\mbox{grad} + (\lambda + \mu)\,\mbox{grad}\, \mbox{div}\,. \enn
Rewriting (\ref{navier})-(\ref{navier2}) into a matrix form, we obtain
\be
\label{13}
\mathbb{L}\bm U=0, \quad \mathbb{L}=\begin{pmatrix}
  (\mu\Delta + \rho \omega ^2)I_2+(\lambda +\mu )\nabla\nabla\cdot & -\gamma \nabla \\
  q\eta \kappa \nabla\cdot & \Delta +q   \end{pmatrix},
  \quad \bm U=\begin{pmatrix} \bm u\\p \end{pmatrix}\,. \en
On $\Gamma$, four different kinds of boundary conditions \cite{K79,C00} taking the forms $\mathbb{B}_i(\pa,\bm\nu)\bm U=\bm G_i$ on $\Gamma$, $i=1,2,3,4$, will be considered.
\begin{itemize}
\item 1st-kind (Dirichlet-type):
\be \label{14}
\mathbb{B}_1(\pa,\bm\nu)=\mathbb{I}_3
\en
\item 2nd-kind (Neumann-type):
\be \label{15}
\mathbb{B}_2(\pa,\bm\nu)=\begin{pmatrix}
  \bm T(\pa,\bm\nu) & -\gamma\bm\nu \\ 0 & \pa_{\bm\nu}  \end{pmatrix} 
\en
\item 3rd-kind (mixed-type):
\be  \label{17} \mathbb{B}_3(\pa,\bm\nu):=\begin{pmatrix}
  \mathbb{I}_2 & 0 \\0 &\pa_{\bm\nu} \end{pmatrix} \en
\item 4th-kind (mixed-type):
  \be  \label{16} 
\mathbb{B}_4(\pa,\bm\nu):=\begin{pmatrix}
  \bm T(\pa,\bm\nu) & -\gamma\bm\nu \\0 & -1 \end{pmatrix}
 \en
\end{itemize}
Here, $\mathbb{I}_n$ denotes a $n\times n$ identity operator, $\bm T(\pa,\bm\nu)$ is the traction operator on the
boundary defined as \ben \bm T(\pa,\bm\nu)\bm u:=2\mu\,\pa_{\bm\nu}\,\bm u+\lambda \bm\nu\,\mbox{div}\bm u
+\mu\bm\nu\times\mbox{curl}\bm u, \qquad \bm\nu=(\nu_1,\nu_2){^\top}, \enn
in which $\bm\nu$ is the outward unit normal to the boundary $\Gamma$
and $\pa_{\bm\nu}=\bm\nu\cdot\nabla$ is the normal derivative. In addition, to ensure the uniqueness of the thermoelastic problem, the scattered field $\bm U$ is assumed to satisfy appropriate radiation conditions at infinity and we refer to \cite[Theorem 3.1]{C00} for the uniqueness of the four thermoelastic open-arc scattering problems.

\subsection{Boundary integral equations}
\label{sec:2.2}

Now we introduce the solution representations and BIEs associated with the considered thermoelastic scattering problems. Let $\mathbb{E}({\bm x},{\bm y})$ be the fundamental solution of the adjoint operator $\mathbb{L}^*$ of $\mathbb{L}$ in $\mathbb{R}^2$ given by
\ben \mathbb{E}(x,y)=\begin{pmatrix}
  \mathbb{E}_{11}({\bm x},{\bm y}) & \bm E_{12}({\bm x},{\bm y}) \\ \bm E_{21}^{\top}({\bm x},{\bm y}) & E_{22}({\bm x},{\bm y}) \end{pmatrix} ,\enn
with
\begin{align*}
&\mathbb{E}_{11}({\bm x},{\bm y})=\frac{1}{\mu}\gamma_{k_s}({\bm x},{\bm y})I+\frac{1}{\rho \omega^2}\nabla_x \nabla_x^{\top}
[\gamma_{k_s}({\bm x},{\bm y})-\frac{k_p^{2}-k_2^{2}}{k_1^{2}-k_2^{2}}\gamma_{k_1}({\bm x},{\bm y})+\frac{k_p^{2}-k_1^{2}}{k_1^{2}-k_2^{2}}\gamma_{k_2}({\bm x},{\bm y})], \\
&\bm E_{12}({\bm x},{\bm y})=\frac{i\omega\eta}{(k_1^{2}-k_2^{2})(\lambda+2\mu)}\nabla_x [\gamma_{k_1}({\bm x},{\bm y})-\gamma_{k_2}({\bm x},{\bm y})], \\
&\bm E_{21}({\bm x},{\bm y})=-\frac{\gamma}{(k_1^{2}-k_2^{2})(\lambda+2\mu)}\nabla_x[\gamma_{k_1}({\bm x},{\bm y})-\gamma_{k_2}({\bm x},{\bm y})], \\
&E_{22}({\bm x},{\bm y})=-\frac{1}{k_1^{2}-k_2^{2}}[(k_p^{2}-k_1^{2})\gamma_{k_1}({\bm x},{\bm y})-(k_p^{2}-k_2^{2})\gamma_{k_2}({\bm x},{\bm y})], \end{align*}
where $\gamma_k ({\bm x},{\bm y})$ denotes the fundamental solution of the Helmholtz equation in $\R^2$ with wavenumber $k$ which takes the form
\be \label{gak}  \gamma_k({\bm x},{\bm y})=\frac{i}{4}H_0^{(1)}(k|{\bm x}-{\bm y}|),\quad x\neq y.  \en
The wavenumbers $k_1$, $k_2$, corresponding to the elastothermal and thermoelastic
waves respectively, are the roots of the characteristic system
\be k_1^{2}+k_2^{2}=q(1+\epsilon )+k_p^{2}, \quad k_1^{2} k_2^{2}=q k_p^{2}, \quad \mu k_s^{2}=\rho \omega^2, \en
for which $k_p:=\omega \sqrt{\rho /(\lambda+2\mu)}$
is the wavenumber of the longitudinal wave in the absence of thermal coupling and
$k_s:=\omega \sqrt{\rho /\mu}$ is the wavenumber of the uncoupled transverse wave. In particular, $\mathrm{Re}\{k_j\}> 0$ and $\mathrm{Im}\{k_j\}> 0$ \cite{WLXY} for $j=1,2$.

Denote by $\mathbb{B}_i^*(\pa,\bm\nu)$ the adjoint operators of $\mathbb{B}_i(\pa,\bm\nu)$ for $i=1,2,3,4$ associated with the Biot system. It can be derived that
\ben
\mathbb{B}_1^*(\pa,\bm\nu)=\mathbb{B}_1(\pa,\bm\nu),\quad \mathbb{B}_2^*(\pa,\bm\nu):=\begin{pmatrix}
  \bm T(\pa,\bm\nu) & -i\omega \eta \bm\nu \\ 0 &\pa_{\bm \nu}
\end{pmatrix},\\
 \mathbb{B}_3^{*}(\pa,\bm\nu)=\mathbb{B}_3(\pa,\bm\nu), \quad
\mathbb{B}_4^{*}(\pa,\bm \nu):=\begin{pmatrix}
  \bm T(\pa,\bm\nu) & -i\omega \eta \bm\nu \\ 0& -1 \end{pmatrix}.
\enn
Then the solution to the open-arc thermoelastic scattering problems can be represented as
\be
\label{solrep}
\bm U(\bm x)=\mathcal{S}_i[\bm\psi_i](\bm x):=\int _{\Gamma}(\mathbb{B}_i^{*}(\pa_{\bm y},\bm\nu_{\bm y})\mathbb{E}({\bm x},{\bm y}))^{\top}\bm \psi_i ({\bm y})d s_{\bm y}, \quad  {\bm x}\in \R^2\backslash \Gamma, \quad i =1,2,3,4,
\en
with the corresponding unknown potentials $\bm\psi_i=(\bm \varphi_i,\psi_i)$ satisfying the BIE
\be
\label{BIE}
\bm V_i[\bm\psi_i]=\bm G_i\quad\mbox{on}\quad \Gamma,
\en
where $\bm V_i$ is a BIO defined as
\be
\label{operatorVi}
\bm V_i[\bm\psi_i](\bm x)=\int _{\Gamma}\mathbb{B}_i(\pa_{\bm x},\bm\nu_{\bm x})(\mathbb{B}_i^{*}(\pa_{\bm y},\bm\nu_{\bm y})\mathbb{E}({\bm x},{\bm y}))^{\top}\bm \psi_i ({\bm y})d s_{\bm y}, \quad  {\bm x}\in \Gamma, \quad i=1,2,3,4.
\en
\tao{Through out of this paper, the strongly- and hyper-singular boundary integrals appearing in $\bm V_i, i=2,3,4$ and the forthcoming BIOs in the following sections are defined in the sense of Cauchy principle value or Hadmard finite part, respectively.}

Assuming that $\Gamma\subset\partial\Omega$ where $\partial\Omega$ is a smooth
boundary of a bounded domain $\Omega$ in $\mathbb{R}^2$, we denote by $\widetilde H^s(\Gamma)$ the space of all $f\in H^s(\partial\Omega)$ such that $\mathrm{supp}(f)\subset\overline{\Gamma}$. Then it can be shown~\cite[Theorem 4.1]{C00} that $\bm V_1: \widetilde H^{-1/2}(\Gamma)^3\rightarrow H^{1/2}(\Gamma)^3$, $\bm V_2: \widetilde H^{1/2}(\Gamma)^3\rightarrow H^{-1/2}(\Gamma)^3$,  $\bm V_3: \widetilde H^{-1/2}(\Gamma)^2\times\widetilde H^{1/2}(\Gamma) \rightarrow H^{1/2}(\Gamma)^2\times H^{-1/2}(\Gamma)$ and $\bm V_
4: \widetilde H^{1/2}(\Gamma)^2\times\widetilde H^{-1/2}(\Gamma) \rightarrow H^{-1/2}(\Gamma)^2\times H^{1/2}(\Gamma)$ are bicontinuous, i.e., continuous and invertible.

\begin{theorem}
Suppose that $\bm G_1\in H^{1/2}(\Gamma)^3$, $\bm G_2\in H^{-1/2}(\Gamma)^3$, $\bm G_3\in H^{1/2}(\Gamma)^2\times H^{-1/2}(\Gamma)$ and $\bm G_4\in H^{-1/2}(\Gamma)^2\times H^{1/2}(\Gamma)$, the four BIEs admit unique solutions in $\widetilde H^{-1/2}(\Gamma)^3$, $\widetilde H^{1/2}(\Gamma)^3$, $\widetilde H^{-1/2}(\Gamma)^2\times\widetilde H^{1/2}(\Gamma)$ and $\widetilde H^{1/2}(\Gamma)^2\times\widetilde H^{-1/2}(\Gamma)$, respectively.
\end{theorem}

\section{Regularized boundary integral equations}
\label{sec:3}

The fact that the boundary integral operator (BIO) $\bm V_1$ is weakly singular, while the operators $\bm V_i$, $i = 2, 3, 4$, are all hypersingular—due to the presence of second-order derivatives in the operators $T(\partial, \bm \nu)$ and/or $\partial_{\bm \nu}$—introduces several challenges:

\begin{itemize} \item[(1)] The eigenvalues of the single-layer operator $\bm V_1$ and the hypersingular operators $\bm V_i$, $i = 2, 3, 4$, form sequences that accumulate at zero and infinity, respectively. As a result, employing Krylov-subspace iterative solvers such as GMRES to solve the corresponding discretized linear systems typically requires a large number of iterations.

\item[(2)] The solutions to the boundary integral equations (BIEs) \eqref{BIE} exhibit singular behavior near the endpoints of the open arc, depending on the type of boundary conditions. This introduces additional difficulties in achieving high accuracy and reducing the number of GMRES iterations in numerical implementations.

\item[(3)] The numerical evaluation of strongly and hypersingular integrals must be properly regularized to ensure computational efficiency. \end{itemize}

Each of these issues will be addressed in the following subsections.

\subsection{Spectrum for closed boundaries}
\label{sec:3.1}

\jp{To reduce the number of GMRES iterations for the solution of (\ref{BIE}), we rely on analytic preconditioning constructed from the Calder\'on identities for the corresponding BIOs. } 

In this subsection, we let $\Gamma=\Gamma_c$ be a smooth closed-curve which is the boundary of a bounded domain $\Omega^-$, i.e., $\Gamma_c=\partial\Omega^-$. The exterior domain $\mathbb{R}^2\backslash\overline{\Omega^-}$ is denoted by $\Omega^+$. Denote
\begin{align*}
& \bm K[\bm \psi](\bm x)= \int _{\Gamma_c}(\mathbb{B}_2^{*}(\pa_{\bm y},\bm\nu_{\bm y})\mathbb{E}({\bm x},{\bm y}))^{\top}\bm \psi ({\bm y})d s_{\bm y}, \quad  {\bm x}\in \Gamma,\\
& \bm K'[\bm \psi](\bm x)= \int _{\Gamma_c}\mathbb{B}_2(\pa_{\bm x},\bm\nu_{\bm x})\mathbb{E}^\top({\bm x},{\bm y})\bm \psi ({\bm y})d s_{\bm y}, \quad  {\bm x}\in \Gamma,\\
& \bm P[\bm\psi](\bm x)=\int _{\Gamma_c}\mathbb{B}_4(\pa_{\bm x},\bm\nu_{\bm x})(\mathbb{B}_3^{*}(\pa_{\bm y},\bm\nu_{\bm y})\mathbb{E}({\bm x},{\bm y}))^{\top}\bm \psi ({\bm y})d s_{\bm y}, \quad  {\bm x}\in \Gamma,\\
& \bm Q[\bm\psi](\bm x)=\int _{\Gamma_c}\mathbb{B}_3(\pa_{\bm x},\bm\nu_{\bm x})(\mathbb{B}_4^{*}(\pa_{\bm y},\bm\nu_{\bm y})\mathbb{E}({\bm x},{\bm y}))^{\top}\bm \psi ({\bm y})d s_{\bm y}, \quad  {\bm x}\in \Gamma.
\end{align*}
Let us introduce the matrix block of the operators, 
$$
\bm{K}  = \begin{pmatrix}
 \bm{K}_{11} & \bm{K}_{12}\\
 \bm{K}_{21} & \bm{K}_{22}
\end{pmatrix}, \bm{K}'  = \begin{pmatrix}
 \bm{K}_{11}' & \bm{K}_{12}'\\
 \bm{K}_{21}' & \bm{K}_{22}'
\end{pmatrix}, \quad 
\bm{V}_i  = \begin{pmatrix}
 \bm{V}_{i,11} & \bm{V}_{i,12}\\
 \bm{V}_{i,21} & \bm{V}_{i,22}
\end{pmatrix}, \ i = 1,2,3,4.
$$

\begin{lemma}
\label{lemma:repPQ}
The matrix blocks of the operators $\bm{V}_3, \bm{V}_4$, $\bm{P}$ and $\bm{Q}$ are given by:
\begin{align*}
    \bm{V}_3  = \begin{pmatrix}
        \bm{V}_{1,11}  & \bm{K}_{12} \\
         \bm{K}'_{21} & \bm{V}_{2,22}
    \end{pmatrix}, \quad
    \bm{V}_4  = \begin{pmatrix}
        \bm{V}_{2,11}  & -\bm{K}'_{12} \\
        - \bm{K}_{21} & \bm{V}_{1,22}
    \end{pmatrix} , 
\end{align*}
and
\begin{align*}
    \bm{P} = \begin{pmatrix}
    \bm{K}'_{11} & \bm{V}_{2,12} \\
    -\bm{V}_{1,21} & -\bm{K}_{22}
    \end{pmatrix}, \quad \bm{Q} = \begin{pmatrix}
    \bm{K}_{11} & -\bm{V}_{1,12} \\
    \bm{V}_{2,21} & -\bm{K}'_{22}
    \end{pmatrix}.
\end{align*}
\end{lemma}
\begin{proof}
The result can be obtained by directly computing the associated singular kernels for each of the BIOs. The details are omitted, as they do not provide meaningful information.
\end{proof}

\jp{

\begin{lemma}
\label{lemma:compactness1}
If $\Gamma_c$ is smooth the following operators are compact.
\begin{align*}
    \bm{V}_{i,12} &: H^{(-1)^i/2}(\Gamma) \rightarrow H^{-(-1)^i/2}(\Gamma)^2, \quad i=1,2, \\
    \bm{V}_{i,21} &: H^{(-1)^i/2}(\Gamma)^2 \rightarrow H^{-(-1)^i/2}(\Gamma), \quad i=1,2, \\ 
    \bm{K}_{12}&: H^{1/2}(\Gamma) \rightarrow H^{1/2}(\Gamma)^2,  \quad \quad
        \bm{K}'_{12}: H^{-1/2}(\Gamma) \rightarrow H^{-1/2}(\Gamma)^2, \\
            \bm{K}_{21}&: H^{1/2}(\Gamma)^2 \rightarrow H^{1/2}(\Gamma),  \quad \quad
        \bm{K}'_{21}: H^{-1/2}(\Gamma)^2 \rightarrow H^{-1/2}(\Gamma), \\
            \bm{K}_{22}&: H^{1/2}(\Gamma) \rightarrow H^{1/2}(\Gamma), \quad \quad
        \bm{K}'_{22}: H^{-1/2}(\Gamma) \rightarrow H^{-1/2}(\Gamma).        
\end{align*}
\end{lemma}
\begin{proof}
For $\bm{K}_{22}$, $\bm{K}'_{22}$ the results is direct consequence of the compactness of the double layer and the adjoint double layer for the Helmholtz equation. All the other cases follows from the fact that the leading singular term of $\bm{E}_{12}$ and $\bm{E}_{21}$ is $(\bm{x}-\bm{y})\log\|\bm{x}-\bm{y}\|$, and standard results for integral operators see, \cite[Chapter 6]{SJG13}. 
\end{proof}
}

\begin{lemma}
\label{lem-jump}
Define
\ben
\mathbb{B}_i^{\pm}\bm v(\bm x)= \lim_{h\rightarrow 0^+, \bm z=\bm x\pm h\bm\nu_{\bm x}} \mathbb{B}_i(\pa_{\bm z},\bm\nu_{\bm x})\bm v(\bm z),\quad x\in\Gamma.
\enn
Then the following jump relations hold:
\begin{align}
\label{jump1}
& \mathbb{B}_2^{\pm} \mathcal{S}_1[\bm\psi]  = \bm{K}'[\bm \psi]\mp \frac{1}{2} \bm \psi,\\
\label{jump2}
& \mathbb{B}_1^{\pm} \mathcal{S}_2[\bm\psi] =  \bm{K}[\bm \psi]\pm \frac{1}{2}\bm \psi,\\
\label{jump3}
& \mathbb{B}_4^{\pm} \mathcal{S}_3[\bm\psi]  = \bm{P}[\bm \psi] \mp \frac{1}{2} \bm \psi,\\
\label{jump4}
& \mathbb{B}_3^{\pm} \mathcal{S}_4[\bm\psi]  = \bm{Q}[\bm \psi] \pm \frac{1}{2}  \bm \psi .
\end{align}
\end{lemma}
\begin{proof}
\jp{The first two identities are the standard relations for the traces of the single-layer and double-layer potentials. The last two follow from the first relations, combined with the jump relations of the potentials and Lemma \ref{lemma:repPQ}. We also refer the reader to \cite[Ch. X, \S 2]{K79} for a detailed analysis.}
\end{proof}

\begin{lemma}
\label{lem-calderon}
For the thermoelastic scattering by closed-surfaces, the Calder\'on identities
\be  
\label{calc}  
\bm V_2\bm V_1=-\frac{\mathbb{I}_3}{4}+(\bm K')^2, \quad \bm V_1\bm V_2=-\frac{\mathbb{I}_3}{4}+\bm K^2,
\en
and
\be
\label{calc2}
\bm V_4\bm V_3=-\frac{\mathbb{I}_3}{4}+\bm P^2, \quad 
\bm V_3\bm V_4=-\frac{\mathbb{I}_3}{4}+\bm Q^2,
\en
hold.
\end{lemma}
\begin{proof}
Let $\bm U$ be a solution to the thermoelastic equation in $\mathbb{R}^2 \backslash \Gamma_c$ that satisfies the Kupradze radiation condition. Denote the jump of $\mathbb{B}_i \bm U$ across the boundary $\Gamma_c$ by $$[\mathbb{B}_i \bm U]=\mathbb{B}_i^- \bm U -\mathbb{B}_i^+ \bm U , \quad i=1,2,3,4.$$ t follows from the solution representation of layer potentials (see \cite{K79}) that \begin{equation} \label{rep} \bm U(\bm x) = \mathcal{S}_1\big[ [\mathbb{B}_2 \bm U] \big] - \mathcal{S}_2\big[ [\mathbb{B}_1 \bm U] \big], \quad \forall, \bm x \in \mathbb{R}^2 \backslash \Gamma_c. \end{equation} By directly evaluating the right-hand side of \eqref{rep}, we obtain an alternative representation: \begin{equation} \label{rep2} \bm U(\bm x) = \mathcal{S}_3\big[ [\mathbb{B}_4 \bm U] \big] - \mathcal{S}_4\big[ [\mathbb{B}_3 \bm U] \big], \quad \forall, \bm x \in \mathbb{R}^2 \backslash \Gamma_c. \end{equation} Using the jump relations stated in Lemma~\ref{lem-jump}, we define the following operators:
\ben
\bm C_1=\frac{\mathbb{I}_3}{2}+
\begin{pmatrix}
-\bm K & \bm V_1\\
-\bm V_2 & \bm K'
\end{pmatrix},\quad \bm C_2=\frac{\mathbb{I}_3}{2}+
\begin{pmatrix}
-\bm Q & \bm V_3\\
-\bm V_4 & \bm P
\end{pmatrix}.
\enn
These operators are projectors, i.e., $\bm C_i^2 = \bm C_i$, for $i = 1,2$. The desired Calderón identities then follow.

\end{proof}

\begin{lemma}
\label{Kspectrum}
It holds that $\bm K^2 - \bm I_{\lambda,\mu}: H^{1/2}(\Gamma)^3\rightarrow H^{1/2}(\Gamma)^3$, $(\bm K')^2 - \bm I_{\lambda,\mu}: H^{-1/2}(\Gamma)^3\rightarrow H^{-1/2}(\Gamma)^3$, $\bm P^2 - \bm I_{\lambda,\mu}: H^{-1/2}(\Gamma)^2\times H^{1/2}(\Gamma)\rightarrow H^{-1/2}(\Gamma)^2\times H^{1/2}(\Gamma)$ and $\bm Q^2 - \bm I_{\lambda,\mu}: H^{1/2}(\Gamma)^2\times H^{-1/2}(\Gamma)\rightarrow H^{1/2}(\Gamma)^2\times H^{-1/2}(\Gamma)$ are compact where 
\ben
\bm I_{\lambda,\mu}=\begin{pmatrix}
 C_{\lambda ,\mu }^2\mathbb{I}_2 & 0\\
 0 & 0
 \end{pmatrix},
\enn
with $C_{\lambda,\mu}$ being a constant given by
\ben
C_{\lambda ,\mu }= \frac{\mu }{2(\lambda  + 2\mu )} < \frac{1}{2}.
\enn
\end{lemma}
\begin{proof}
Note that the thermoelastic scattering problem takes an analogous Biot system as the poroelastic scattering problem~\cite{S09,SSU09}. Then the compactness of $(\bm K)^2 - \bm I_{\lambda,\mu}$ and $(\bm K')^2 - \bm I_{\lambda,\mu}$ can be shown by following the proof of \cite[Theorem 3.1]{ZXY21}. To show the compactness of $(\bm P)^2 - \bm I_{\lambda,\mu}$, we note from the matrix blocks of $\bm P$ that
\ben
(\bm P)^2 - \bm I_{\lambda,\mu}= \begin{pmatrix}
(\bm{K}'_{11})^2-\bm{V}_{2,12}\bm{V}_{1,21} -C_{\lambda ,\mu }^2\mathbb{I}_2 & \bm{K}'_{11}\bm{V}_{2,12}-\bm{V}_{2,12}\bm{K}_{22}\\ -\bm{V}_{1,21}\bm{K}'_{11}+ \bm{K}_{22}\bm{V}_{1,21} & -\bm{V}_{1,21}\bm{V}_{2,12}+\bm{K}_{22}^2
\end{pmatrix},
\enn
\jp{the result is then a direct consequence of the compactness of the first component of $(\bm K')^2 - \bm I_{\lambda,\mu}$, and Lemma \ref{lemma:compactness1}. The proof for $(\bm Q)^2 - \bm I_{\lambda,\mu}$ is essentially the same. } 
\end{proof}


\begin{theorem}
\label{V12}
It follows that $\bm V_2\bm V_1: H^{-1/2}(\Gamma)^3\rightarrow H^{-1/2}(\Gamma)^3$ \jp{,} $\bm V_1\bm V_2: H^{1/2}(\Gamma)^3\rightarrow H^{1/2}(\Gamma)^3$, \jp{$\bm V_4\bm V_3: H^{-1/2}(\Gamma)^2 \times H^{1/2}(\Gamma) \rightarrow H^{-1/2}(\Gamma)^2 \times H^{1/2}(\Gamma) $, $\bm V_3\bm V_4: H^{1/2}(\Gamma)^2 \times H^{-1/2}(\Gamma) \rightarrow H^{1/2}(\Gamma)^2 \times H^{-1/2}(\Gamma) $} are second-kind Fredholm operators and their spectrum consist of two nonempty sequences of eigenvalues which accumulate at $-\frac{1}{4}$ and $-\frac{1}{4}+C_{\lambda,\mu}^2 (<0)$, respectively.
\end{theorem}
\begin{proof}
\jp{This is a direct consequence of lemmas \ref{lem-calderon}, and \ref{Kspectrum}}.

\end{proof}

\subsection{Edge singularities and regularized BIEs}
\label{sec:3.2}

It is known that the solutions to the problem of wave scattering by open-surfaces exhibit certain types of singular behaviour near the edge of the open-surface (two endpoints of the two-dimensional open-arc). As shown in~\cite[Theorems 4.2-4.5]{C00}, the solutions $\bm\psi_i=(\bm\phi_i,\psi_i)$, $i=1,2,3,4$ to the BIEs (\ref{BIE}) can be expressed asymptotically as
\ben
\bm\phi_1\sim \frac{\bm\xi_1}{\sqrt{d}}+\bm\chi_1,\quad \psi_1\sim \frac{\zeta_1}{\sqrt{d}}+\varrho_1,\\
\bm\phi_2\sim \bm\xi_2\sqrt{d}+\bm\chi_2,\quad \psi_2\sim \zeta_2\sqrt{d}+\varrho_2,\\
\bm\phi_3\sim \frac{\bm\xi_3}{\sqrt{d}}+\bm\chi_3,\quad \psi_3\sim \zeta_3\sqrt{d}+\varrho_3, \\
\bm\phi_4\sim \bm\xi_4\sqrt{d}+\bm\chi_4,\quad \psi_4\sim \frac{\zeta_4}{\sqrt{d}}+\varrho_4,
\enn
where for $i=1,2,3,4$, $\bm \xi_i,\zeta_i$ denote certain cut-off functions, $\bm\chi_i,\varrho_i$ are less singular than $\bm\phi_i,\psi_i$, and $d$ denotes the distance to the endpoints of the open-arc $\Gamma$. Assuming the boundary $\Gamma$ and boundary data $\bm G_i$, $i=1,2,3,4$ to be sufficiently smooth, we can simply express the density functions $\bm\psi_i=(\bm\phi_i,\psi_i)$, $i=1,2,3,4$ as 
\ben
\bm\phi_1= \frac{\bm\xi_1}{w},\quad \psi_1= \frac{\zeta_1}{w},\\
\bm\phi_2= \bm\xi_2w,\quad \psi_2= \zeta_2w,\\
\bm\phi_3= \frac{\bm\xi_3}{w},\quad \psi_3= \zeta_3{w},\\
\bm\phi_4= \bm\xi_4 w,\quad \psi_4= \frac{\zeta_4}{w},
\enn
where $w$ is a weight function that reproduces the $\sqrt{d}$ asymptotic character. Denoting 
\ben
\mathbb{W}_1=\mathrm{diag}\{w^{-1},w^{-1},w^{-1}\}, \quad \mathbb{W}_2=\mathrm{diag}\{w,w,w\},\\
 \mathbb{W}_3=\mathrm{diag}\{w^{-1},w^{-1},w\},\quad
\mathbb{W}_4=\mathrm{diag}\{w,w,w^{-1}\},
\enn
we define the weighted BIOs associated with $\bm V_i$, $i=1,2,3,4$ as
\be
\bm V_i^w[\bm\psi]= \bm V_i[\mathbb{W}_i\bm\psi].
\en
Then the original BIEs (\ref{BIE}) can be replaced by 
\be
\label{weighted}
\bm V_i^w[\bm\psi_i](\bm x)= \bm G_i(\bm x),\quad \bm x\in\Gamma,
\en
which explicitly incorporates the edge singularities of the unknown densities.

Similarly to the case of the standard BIOs ($\bm V_i$, $i=1,2,3,4$), which was  mentioned at the beginning of Section~\ref{sec:3}, it can be shown that the eigenvalues of the single-layer operator $\bm V_1^w$ and hyper-singular operators $\bm V_i^w$, $i=2,3,4$ consist of sequences which accumulate at zero and/or infinity depending on the singularity of the kernels. To reduce the required numbers of GMRES iterations, appropriate preconditioning or regularization techniques are necessary. It has been shown in~\cite{XY22} that the weighted elastic BIOs on a smooth open-arc enjoy an analogous spectral property as the classical elastic BIOs on a smooth closed-surface. 

Then taking in considerations the results from Section~\ref{sec:3.1} we modify the solution representation (\ref{solrep}) and
construct regularized BIEs to solve the thermoelastic open-arc scattering problems. The regularized solution representations and BIEs for the four cases of boundary conditions are listed in Table~\ref{RsolrepHIE}. 

\begin{table}[htb]
\label{modified}
  \caption{The regularized solution representations and BIEs.}
\centering
\begin{tabular}{c|c|c}
\hline
type of boundary condition & solution representation in $\mathbb{R}^2\backslash\Gamma$ & BIE on $\Gamma$ \\
\hline
1st-kind & $\bm U=\mathcal{S}_1^w[\widetilde{\bm\psi}_1]$  &  $\bm V^w_2 \bm V_1^w[\widetilde{\bm\psi}_1]=\bm V^w_2[\bm G_1]$ \\
\hline
2nd-kind & $\bm U=\mathcal{S}_2^w\bm V_1^w[\widetilde{\bm\psi}_2]$  &  $\bm V^w_2 \bm V_1^w[\widetilde{\bm\psi}_2]=\bm G_2$\\
\hline
3rd-kind & $\bm U=\mathcal{S}_3^w[\widetilde{\bm\psi}_3]$  &  $\bm V_4^w\bm V_3^w[\widetilde{\bm\psi}_3]=\bm V_4^w[\bm G_3]$\\
\hline
4th-kind & $\bm U=\mathcal{S}_4^w \bm V_3^w[\widetilde{\bm\psi}_4]$  &  $\bm V_4^w\bm V_3^w[\widetilde{\bm\psi}_4]= \bm G_4$ \\
\hline
\end{tabular}
\label{RsolrepHIE}
\end{table}

\subsection{Spectrum for open arcs}
\label{sec:3.3}

The purpose of this section is to study the spectral properties of the regularized formulations presented in Section~\ref{sec:3.2}. Without loss of generality, we assume that the open-arc $\Gamma$ is analytic and can be parameterized by means of a smooth vector function $\bm x=\bm x(t)=(x_1(t),x_2(t))^\top\in\Gamma, t\in(-1,1)$ with $\mathcal{J}(t)=|\bm x'(t)|\ne 0$ in which the prime $'$ denotes the derivative with respect to $t$. In this case, the weight function $w$ can be selected as $w(\bm x(t))=\sqrt{1-t^2}$ to produce the $\sqrt{d}$ asymptotic character. It is important to highlight that the spectral properties of the composite operators $\bm{J}^{2,1}=\bm V^w_2 \bm V_1^w$ and $\bm{J}^{4,3}=\bm V^w_4 \bm V_3^w$ do not come from a direct consequence of the results for closed boundaries (Theorem \ref{V12}), as formulas \eqref{rep} and \eqref{rep2} are no valid for open boundaries.

We begin with the discussion of line-segment case $\Gamma=\Gamma_0$ with $x_1(t)=t, x_2(t)=0, t\in(-1,1)$ for $x\in\Gamma_0$ which primarily rely on the tools introduced in \cite{pinto2024shape} and the results for general case follows analogously. Let 
$T_n$ and $U_n$ denote the first-kind and second-kind Chebyshev polynomials, respectively. We will utilize the following Chebyshev-based function spaces:
$$
T^s := \left\lbrace u = \sum_{n=0}^\infty a_n \frac{T_n}{w}: \|u\|_{T^s}^2 = \sum _{n=0}^\infty (1+n^2)^s |a_n|^2 < \infty \right\rbrace,
$$
$$
W^s :=  \left\lbrace u = \sum_{n=0}^\infty a_n T_n: \|u\|_{W^s}^2 = \sum _{n=0}^\infty (1+n^2)^s |\widehat{a}_n|^2 < \infty\right\rbrace,
$$
$$
U^s :=  \left\lbrace u = \sum_{n=0}^\infty a_n w U_n: \|u\|_{U^s}^2 = \sum _{n=0}^\infty (1+n^2)^s |\ddot{a}_n|^2 < \infty \right\rbrace,
$$
$$
Y^s :=  \left\lbrace u = \sum_{n=0}^\infty a_n U_n: \|u\|_{Y^s}^2 = \sum _{n=0}^\infty (1+n^2)^s |\widetilde{a}_n|^2 < \infty \right\rbrace,
$$
where $s \in \mathbb{R}$, and  
$$a_n = \frac{2-\delta_{0n}}{\pi} \int_{-1}^1 u(t) T_n(t) dt, \quad  \widehat{a}_n = \frac{2-\delta_{0n}}{\pi} \int_{-1}^1 u(t) \frac{T_n}{w}(t) dt,$$
$$\ddot{a}_n = \frac{2}{\pi} \int_{-1}^1 u(t) U_n(t) dt, \quad \widetilde{a}_n = \frac{2}{\pi} \int_{-1}^1 u(t) wU_n(t) dt.$$
Here, we use the notation $u(t)=u(\bm x(t))$ for $\bm x\in\Gamma_0$. Furthermore, the Chebyshev-based spaces are related to the classical Sobolev spaces in the following way, 
$$
\widetilde{H}^{-1/2}(\Gamma_0) = T^{-1/2}, \quad H^{-1/2}(\Gamma_0) = Y^{-1/2}, 
$$
$$
\widetilde{H}^{1/2}(\Gamma_0) = U^{1/2}, \quad H^{1/2}(\Gamma_0) = W^{1/2}.
$$
We will also make use of some properties of the Chebyshev polynomials that are collected in the following result. 
\begin{proposition}
\label{cheb:prop1}
The follows relations hold:
    $$T_n ' = nU_{n-1},\quad (wU_n)' = -(n+1)\frac{T_{n+1}}{w},   $$
$$T_n = \frac{1}{2}(U_n - U_{n-2}), \,\, n\geq 2 \quad\mbox{and}\quad 
U_n =
2 \sum_{j \ \text{odd/even}}^n T_j, $$
where the sum in the last equation means that $j$ takes all the odd numbers from $1$ to $n$ if $n$ is odd, or all the even numbers from $0$ to $n$ if $n$ is even, and the term $j=0$ is divided by $2$.  
\end{proposition}

The key property of Chebyshev polynomials in this context is that they serve as eigenfunctions of the following operators \cite{YS88}.
\begin{proposition}
\label{cheb:prop2}
It holds that for $n\in\mathbb{N}$,
\begin{align}
   & \int_{-1} ^1 \frac{-1}{2\pi} \log|t-s| w^{-1}(s) T_n(s) ds  = d_n T_n, \quad  d_n = \begin{cases}
	\frac{\log 2}{2}, \quad n=0, \\ 
	\frac{1}{2n}, \quad n > 0,
\end{cases}
\\
    & \frac{d}{dt} \int_{-1} ^1 \frac{-1}{2\pi} \log|t-s|\frac{d}{ds} (w(s) U_n(s)) ds  = e_n U_n, \quad  e_n = -\frac{n+1}{2}.
\end{align}
\end{proposition}

To study the spectral properties of $\bm{J}^{2,1}$ and $\bm{J}^{4,3}$, we recall some spectral results for open arc problems of Helmholtz equation presented in \cite{LB15}. The integral operators for this case are defined as follows:
$$
(V_1^H u)(\bm x) = \int_{\Gamma_0} \gamma_k(\bm x,\bm y) u(\bm y) d\bm y,
$$
$$
(V_2^H u) (\bm x)  =  \int_{\Gamma_0} \partial_{\bm\nu_{\bm x} }\partial_{\bm{\nu}_{\bm y}} \gamma_k(\bm{x},\bm y )   u(\bm y) d\bm y,
$$
where $k >0$. The corresponding weighed versions are, $V_1^{H,w}  = V_1^H \circ w^{-1}$, and $V_2^{H,w}  = V_2^H \circ w^{}$. 
\begin{proposition}
\label{prop:helmspectrum}
    The pointwise spectrum of the bounded operator
    $$
     J^H := V_2^{H,w}V_1^{H,w} : W^s \rightarrow W^s, \quad s > -\frac{1}{2},
    $$
is the union of a discrete part $(\lambda_n)_{n \in \mathbb{N}}$, where 
$$
\lambda_n  = \begin{cases}
    -\frac{\log 2}{4}, \quad n=0, \\
    -\frac{1}{4}-\frac{1}{4n}, \quad n >0,  
\end{cases} 
$$
and a complementary part
$$
\left\lbrace
\lambda_x+i \lambda_y: \lambda_x +\frac{1}{4} = r \cos \theta ,\ \lambda_y= r \sin \theta, \ \ 0<r<-\frac{\cos\theta}{4s+2}, \, \frac{\pi}{2}<\theta< \frac{3\pi}{2}  
\right\rbrace ,
$$
up to a compact pertubation.
\end{proposition}

We also present the pointwise spectrum of $\widehat{J}^H := V_1^{H,w} V_2^{H,w}$, which plays an important role in the spectral properties of $\bm V_4^w\bm V_3^w$.

\begin{proposition}
\label{prop:helmspectrum2}
    The pointwise spectrum of 
$
\widehat{J}^H :W^{s+1} \to W^{s+1},\,s >-\frac{1}{2}$
  coincides with that of 
$J^H$, up to a compact perturbation.
\end{proposition}
\begin{proof}
Arguing as in \cite{LB15} we only need to consider the principal parts of $V_1^{H,w}$ and $V_2^{H,w}$. Hence we are left with the following operator 
\begin{equation}
\widehat{J}^H_0  = \left(  \int_{-1}^1 \frac{-1}{2\pi}\log|t-s| (w(s)^{-1}  \cdot ) ds \right)
\circ
\left( \frac{d}{dt} \int_{-1}^1 \frac{-1}{2\pi}\log|t-s| \frac{d}{ds}(w(s)  \cdot ) ds \right).
\end{equation}
A direct computation using propositions \ref{cheb:prop1} and \ref{cheb:prop2} shows that 
\begin{equation}
\widehat{J}^H_0  T_n  = e_n d_n T_n +(e_n-e_{n-2})\sum_{j \text{ odd/even}}^{n-2} d_j T_j,  
\end{equation}
where the sum on the right-hand vanishes when $n=0,1$. Thus, $\widehat{J}^H_0$ is an upper triangular operator and its diagonal terms  $\{e_n d_n\}_{n \in \mathbb{N}}$, are eigenvalues. On the other hand if we have any other eigenfunction $f = \sum_{n=0}^{\infty} f_n T_n \in W^{s+1}$, we will have that 
\begin{equation}
\widehat{J}^H_0 f = \sum_{n=0}^{\infty} f_n \left(e_n d_n T_n +(e_n-e_{n-2})\sum_{j \text{ odd/even}}^{n-2} d_j T_j\right ) = z\sum_n f_n T_n,
\end{equation}
where $z$ is the corresponding eigenvalue. From the last equation, we deduce that
\begin{equation}
f_n e_n d_n  + \sum_{j=0}^\infty f_{n+2(j+1)} (e_{n+2(j+1)}- e_{n+2j})d_n = z f_n.
\end{equation}
Hence, by combining the expressions for $zf_n$ and $zf_{n+2}$, we obtain 
\begin{equation}
\frac{z}{d_n} f_n - \frac{z}{d_{n+2}}f_{n+2} = e_n(f_n-f_{n+2}).
\end{equation}
Thus, 
\begin{equation}
(n+2)f_{n+2} = (n+2)\left(\frac{\frac{z}{d_n}-e_n}{\frac{z}{d_{n+2}}-e_n}\right)f_n = \left(\frac{z+\frac{1}{4}(1+\frac{1}{n})}{z+\frac{1}{4}(1-\frac{1}{n+2})}\right)(n f_n), \quad n \geq 1.
\end{equation}
Following a proof similar to that of \cite[Lemma 9]{LB15}, the absolute value of $n f_n$ is asymptotically given by 
\begin{equation}
|n f_n|=O(n^{\frac{-2x}{x^2+y^2}}),
\end{equation}
where $x,y \in \mathbb{R}$ satisfying $-x+yi=2(4z+1),\, x>0$. Thus, our conclusion is obtained from $\sum_{n=0}^{\infty} (1+n^2)^{s+1} |f_n|^2 < \infty$.
\end{proof}

Now we are ready to present the equivalent result of Proposition \ref{prop:helmspectrum} for the thermoelastic operators on $\Gamma_0$.
\begin{theorem}
\label{v2v1thermo}
Let $\Gamma=\Gamma_0$. The pointwise spectrum of $\bm{J}^{2,1} : (W^s)^3 \rightarrow (W^s)^3$
   is the union of two sets (up to a compact pertubation), namely $\Lambda^1, \Lambda^2 $, each of them containing a discrete part (denoted by $\Lambda^1_D$, and $\Lambda^2_D$ respectably), and a complementary part (denoted by $\Lambda^1_C$, and $\Lambda^2_C$). Furthermore we have that
   $$
   \Lambda^1_D=C^{(1)}_{\lambda,\mu}C^{(2)}_{\lambda,\mu}\left(\left\lbrace -\frac{\log 2}{4}\right \rbrace\cup \left\lbrace  -\frac{1}{4} - \frac{1}{4n}\right \rbrace_{n \in \mathbb{N}_{>0}} \right), 
   $$
   $$
   \Lambda^2_D=\left\lbrace -\frac{\log 2}{4}\right \rbrace\cup \left\lbrace  -\frac{1}{4} - \frac{1}{4n}\right \rbrace_{n \in \mathbb{N}_{>0}}, 
   $$
   and 
   $$\Lambda^1_C = C^{(1)}_{\lambda,\mu}C^{(2)}_{\lambda,\mu} \left\lbrace
\lambda_x+i \lambda_y: \lambda_x +\frac{1}{4} = r \cos \theta ,\ \lambda_y= r \sin \theta, \ \ 0<r<-\frac{\cos\theta}{4s+2}, \, \frac{\pi}{2}<\theta< \frac{3\pi}{2}  
\right\rbrace ,$$
   $$\Lambda^2_C = \left\lbrace
\lambda_x+i \lambda_y: \lambda_x +\frac{1}{4} = r \cos \theta ,\ \lambda_y= r \sin \theta, \ \ 0<r<-\frac{\cos\theta}{4s+2}, \, \frac{\pi}{2}<\theta< \frac{3\pi}{2}  
\right\rbrace,  $$
where, 
$$C^{(1)}_{\lambda,\mu} = \frac{\lambda +3\mu}{2\mu(\lambda + 2\mu)}, \quad C^{(2)}_{\lambda,\mu} = \frac{2\mu(\lambda +\mu) }{\lambda +2 \mu}.$$
\end{theorem}

\begin{proof}
    From the definition of the fundamental solution, the principal part of $\bm{V}^w_1$ gives rise to the following operator, 
\begin{align*}
    \bm{V}^{S,w}_1 \begin{pmatrix} \bm{\psi}\\ \varphi\end{pmatrix} (\bm{x})= \begin{pmatrix}
        C_{\lambda, \mu }^{(1)}\int_{\Gamma_0}\frac{-1}{2\pi} \log \|\bm x - \bm y \| (w^{-1}(\bm y) \cdot) \mathbb{I}_2 d\bm y & \bm{0} \\ 
        \bm{0}^\top  & \int_{\Gamma_0} \frac{-1}{2\pi }\log \|\bm x - \bm y \| (w^{-1}(\bm y) \cdot) d \bm y
    \end{pmatrix}
    \begin{pmatrix} \bm{\psi}\\ \varphi\end{pmatrix},
\end{align*}    
on the other hand, from the formulas in Theorem \ref{regV3} we have that the principal part of $\bm{V}^w_2$ gives rise to the following operator 
\begin{align*}
&\bm{V}^{S,w}_2 \begin{pmatrix} \bm{\psi}\\ \varphi\end{pmatrix} (\bm{x})\\
= &\begin{pmatrix}
        C_{\lambda, \mu }^{(2)}D_{s,\bm x}\int_{\Gamma_0}\frac{-1}{2\pi} \log \|\bm x - \bm y \|   D_{s,\bm y}(w(\bm y)\cdot ) \mathbb{I}_2 d\bm y & \bm{0} \\ 
        \bm{0}^\top  & D_{s,\bm x}\int_{\Gamma_0} \frac{-1}{2\pi }\log \|\bm x - \bm y \| D_{s,\bm y} (w(\bm y) \cdot) d \bm y
    \end{pmatrix}
    \begin{pmatrix} \bm{\psi}\\ \varphi\end{pmatrix},   
\end{align*}    
where $D_{s,\cdot}$ denotes the tangential derivative on the appropriate variable. Then we obtain that the spectrum of $\bm J^{2,1}$ is given (up to compact perturbations)  by the union of the spectrums of 
$$
\bm{j}_{11}^{2,1} =  \left(C_{\lambda, \mu }^{(2)}D_{s,\bm x}\int_{\Gamma_0}\frac{1}{2\pi} \log \|\bm x - \bm y \|  D_{s,\bm y}(w(\bm y ) \cdot) \mathbb{I}_2 d\bm y   \right)\circ  \left(   C_{\lambda, \mu }^{(1)}\int_{\Gamma_0}\frac{-1}{2\pi} \log \|\bm x - \bm y \| (w^{-1} (\bm y)\cdot )\mathbb{I}_2 d\bm y\right),
$$
$$
\bm{j}_{22}^{2,1} =  \left(D_{s,\bm x}\int_{\Gamma_0}\frac{1}{2\pi} \log \|\bm x - \bm y \|  D_{s,\bm y} (w(\bm y)\cdot )d\bm y   \right)\circ  \left(   \int_{\Gamma_0}\frac{-1}{2\pi} \log \|\bm x - \bm y \| (w^{-1}(\bm y) \cdot)d\bm y\right).
$$
The case of $\bm{j}^{2,1}_{22}$ is identical to the Helmholtz case. Thus, $\Lambda^2$ is obtained from Proposition \ref{prop:helmspectrum}. Similarly, the spectrum of $\bm{j}^{2,1}_{22}$ also follows from Proposition \ref{prop:helmspectrum}, as the operators are the same but with an scaling factor $C^{(1)
}_{\lambda, \mu} C^{(2)}_{\lambda , \mu }$.
\end{proof}

\begin{corollary}
\label{v4v3thermo}
Let $\Gamma=\Gamma_0$. The pointwise spectrum of 
$\bm{J}^{4,3}:(W^s)^2 \times W^{s+1} \to (W^s)^2 \times W^{s+1}$
  coincides with that of 
$\bm{J}^{2,1}$ up to a compact perturbation.
\end{corollary}
\begin{proof}
From lemmas \ref{lemma:repPQ} and \ref{lemma:compactness1}, we know that the off-diagonal terms $\bm{V}^w_3$ and $\bm{V}^w_4$ are compact. Thus, we only need to consider the spectrum of 
$\bm{V}_{2,11} \circ w \circ \bm{V}_{1,11}\circ w^{-1}$ and $\bm{V}_{1,22} \circ w^{-1} \circ \bm{V}_{2,22} \circ w$. The first one is equal to the first block of $\bm{J}^{2,1}$ whose spectrum is $\Lambda_1$, while the second one is obtained from Proposition \ref{prop:helmspectrum2}.
\end{proof}

It can be seen from Theorem \ref{v2v1thermo} and Corollary \ref{v4v3thermo} that the point spectra of the operators $\bm{V}_2^w \bm{V}_1^w$ and $\bm{V}_4^w \bm{V}_3^w$ are bounded away from zero and infinity. In addition, the discrete spectrum $\Lambda_D^1$ and $\Lambda_D^2$ have two cluster points, $-\frac{1}{4}$ and 
$$-\frac{1}{4}C_{\lambda,\mu}^{(1)}C_{\lambda,\mu}^{(2)}=-\frac{1}{4}[1-(\frac{\mu}{\lambda+2\mu})^2]=-\frac{1}{4}+C_{\lambda,\mu}^2,$$
which are exactly the same as those in the closed-surface case in Theorem \ref{V12}. Utilizing the regularized formulations in Section \ref{sec:3.3} and the Nystr\"om methodology in Section \ref{sec:4.1}, we have verified these spectral properties numerically. Figure \ref{V4V3_eig} displays the eigenvalue distributions of the operator $\bm{V}_4 \bm{V}_3$ (defined on the unit circle) and the operator $\bm{V}_4^w \bm{V}_3^w$ (defined on $\Gamma_0$) for different values of $\lambda$, where the parameters $\omega=50,\rho=1,\mu=1,\kappa=1, \eta=0.2, \gamma=0.1$ are selected.  Note that, in all cases, the eigenvalues of $\bm{V}_4 \bm{V}_3$ and ${\bm{V}_4^w \bm{V}_3^w}$  accumulate
 around $(-\frac{1}{4}+C_{\lambda,\mu}^2 ,0)$ and $(-\frac{1}{4},0)$, as predicted by previous results. The operators $\bm{V}_2 \bm{V}_1$ and $\bm{V}_2^w \bm{V}_1^w$ have similar spectrum distributions so we omit the figures here. 

\begin{remark}
For general smooth open-arc $\Gamma$ parameterized as $\bm x=\bm x(t)=(x_1(t),x_2(t))^\top\in\Gamma, t\in(-1,1)$ with $\mathcal{J}(t)=|\bm x'(t)|\ne 0$, it can be shown following \cite[Theorem 4.3]{XY22} that the pointwise spectrum of $\bm{J}^{2,1}$ and $\bm{J}^{4,3}$ coincides with that in Theorem~\ref{v2v1thermo} up to a compact perturbation. We prefer to omitting the details here.
\end{remark}


\begin{figure}[htbp]\small
\centering
\begin{tabular}{cc}
\includegraphics[scale=0.3]{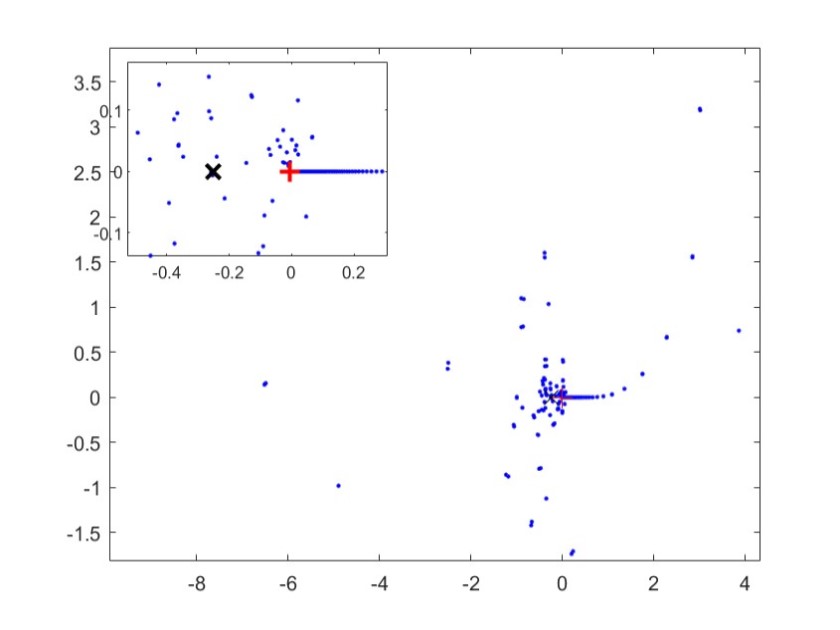} &
\includegraphics[scale=0.3]{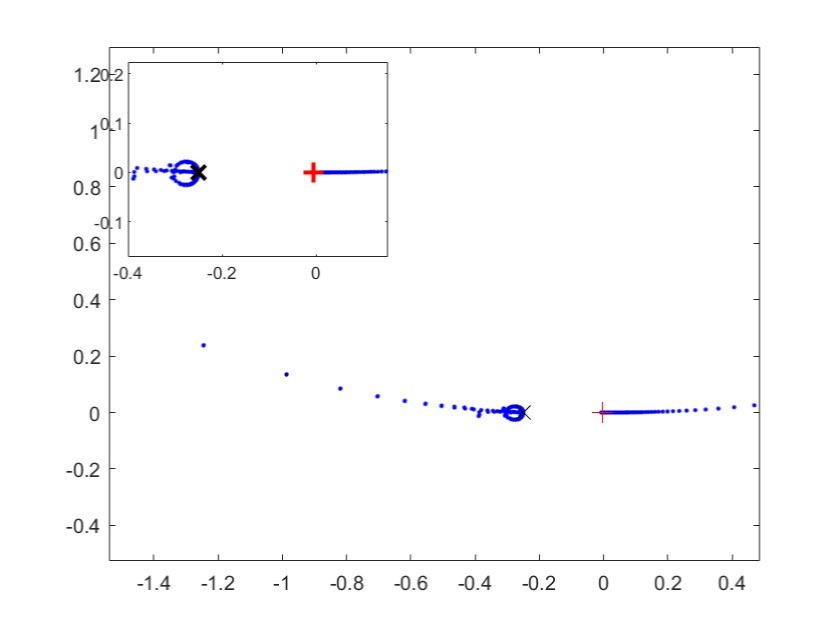} \\
(a) $eig(\bm V_4\bm V_3),\lambda=-0.99$ &
(b) $eig(\bm V_4^w\bm V_3^w),\lambda=-0.99$ \\
\includegraphics[scale=0.3]{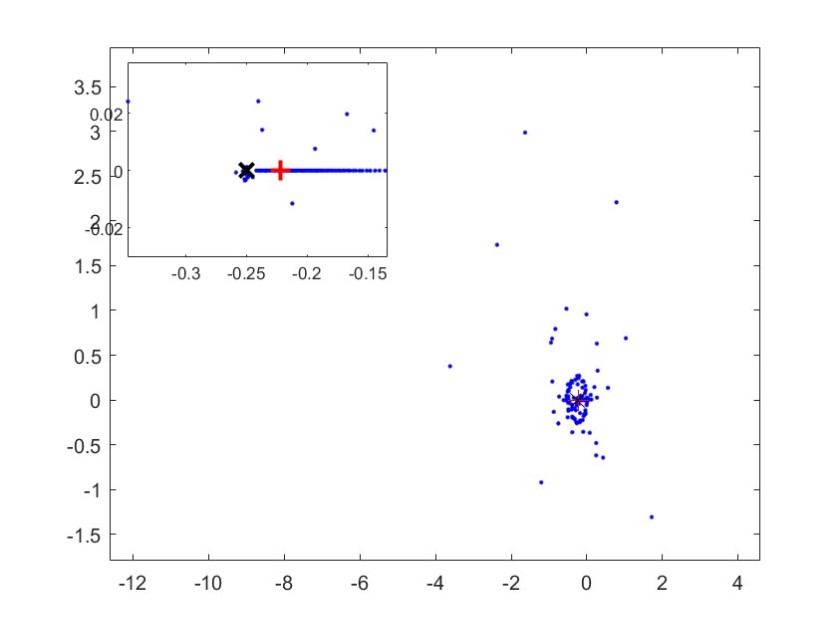} &
\includegraphics[scale=0.3]{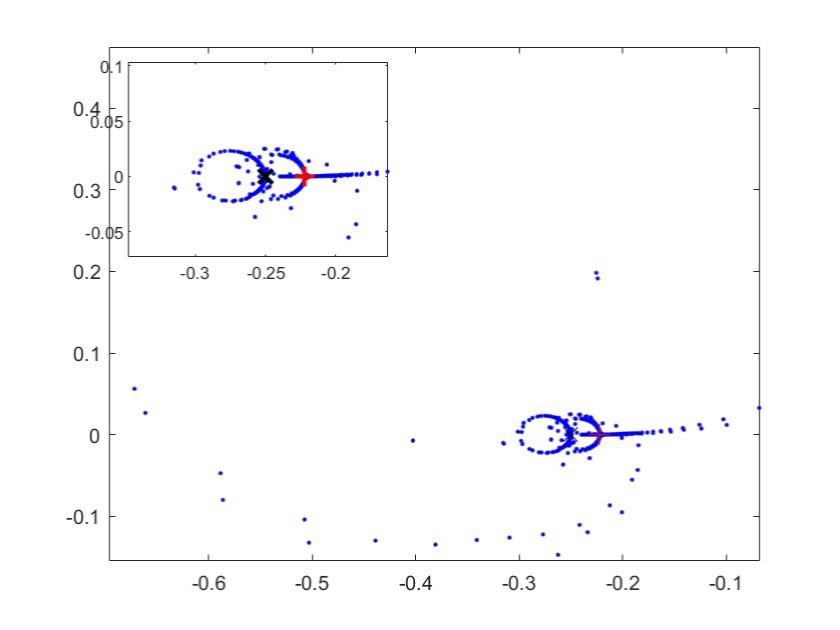}\\
(c) $eig(\bm V_4\bm V_3),\lambda=1$ & 
(d) $eig(\bm V_4^w\bm V_3^w),\lambda=1$ \\
\includegraphics[scale=0.3]{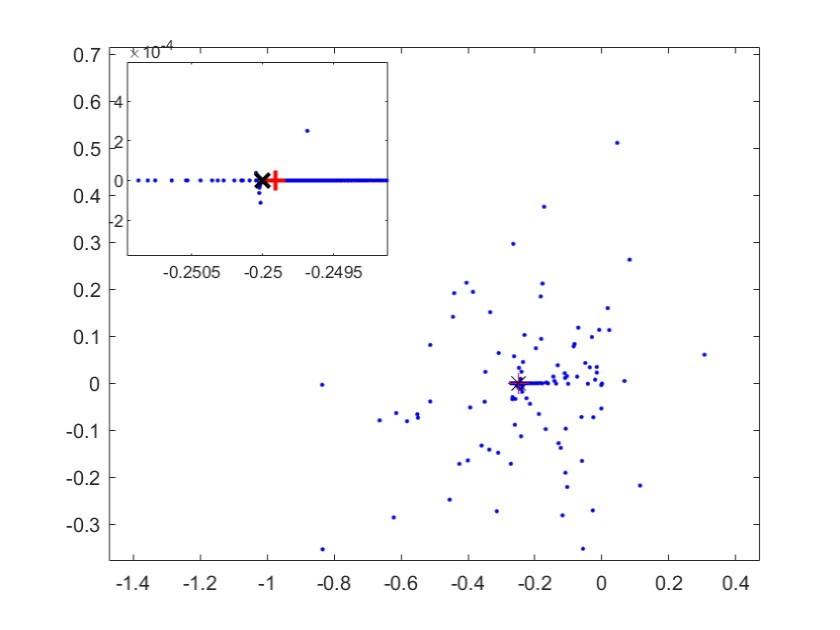} &
\includegraphics[scale=0.3]{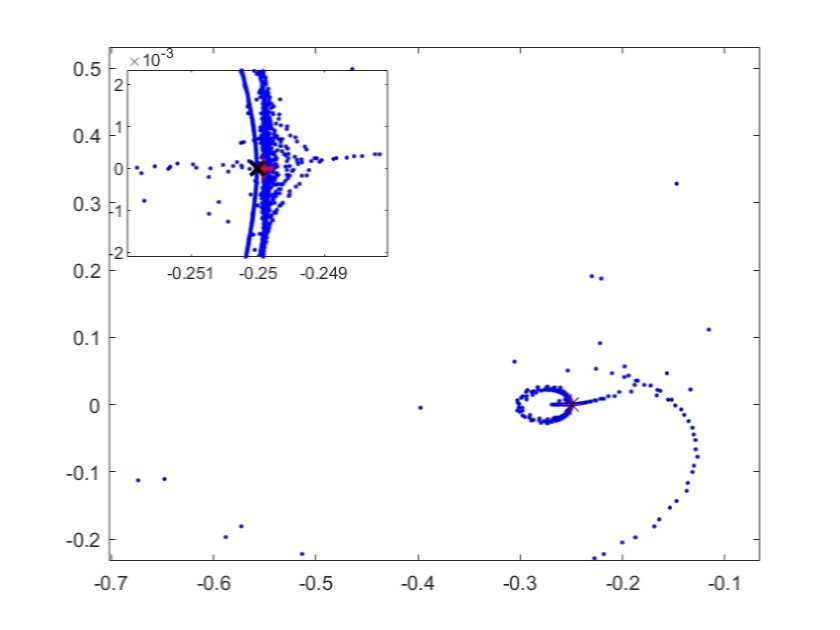} \\ 
(e) $eig(\bm V_4\bm V_3),\lambda=50$ &
(f) $eig(\bm V_4^w\bm V_3^w),\lambda=50$
\end{tabular}
\caption{Eigenvalue distribution of $\bm V_4\bm V_3$ and $\bm V_4^w\bm V_3^w$ for different values of $\lambda$. $(+): (C_{\lambda,\mu}^2-\frac{1}{4},0)$; $(\times): (-\frac{1}{4},0)$.}
\label{V4V3_eig}
\end{figure}

\subsection{Regularization of Singular Integrals}
\label{sec:3.4}
Now we are ready to study the regularization techniques for the treatment of strongly- and hyper-singular integrals in numerical evaluation. To avoid repetitive and redundant deductions, here we only present the result for the operator $\bm V_4^w$. The regularized formulations for $\bm V_i^w$, $i=2,3$ can be derived in an analogous manner and we present the details in the appendix.

Denoting 
\ben
\bm V_i^w[\bm\psi]=\begin{pmatrix}
\bm V_{i,11}^w & \bm V_{i,12}^w\\
\bm V_{i,21}^w & V_{i,22}^w
\end{pmatrix}\begin{pmatrix}
\bm\phi \\ \psi
\end{pmatrix}, \quad i=2,3,4,
\enn
it follows that
\begin{align*}
\bm V_{4,11}^w[\bm\phi](x)=& \int_\Gamma \big[\bm T(\pa_{\bm x},\bm\nu_{\bm x})(\bm T(\pa_{\bm y},\bm\nu_{\bm y})\mathbb{E}_{11}(\bm x,\bm y))^{\top}-i\omega\eta \bm T(\pa_{\bm x},\bm\nu_{\bm x})\bm E_{21}(\bm x,\bm y)\bm\nu_{\bm y}^{\top}   \\
&\,-\gamma {\bm\nu}_{\bm x}(\bm T(\pa_{\bm y},\bm\nu_{\bm y})\bm E_{12}(\bm x,\bm y))^{\top}+i\omega\eta\gamma\bm\nu_{\bm x} \bm \nu_{\bm y}^{\top}E_{22}(\bm x,\bm y)\big] w(\bm y) \bm \phi(\bm y)ds_{\bm y},\quad  \bm x\in\Gamma,\\
\bm V_{4,12}^w[\psi](x)=& \int_\Gamma \left[-\bm T(\pa_{\bm x},\bm\nu_{\bm x}) \bm E_{21}(\bm x,\bm y) +\gamma \bm\nu_{\bm x} E_{22}(\bm x,\bm y)\right]\frac{\psi(\bm y)}{w(\bm y)}ds_{\bm y} ,\quad \bm x\in\Gamma,\\
\bm V_{4,21}^w[\bm\phi](x)=& \int_\Gamma \left[-(\bm T(\pa_{\bm y},\bm\nu_{\bm y})\bm E_{12}(\bm x,\bm y))^{\top}+i\omega\eta \bm \nu_{\bm y}^{\top}E_{22}(\bm x,\bm y)\right] w(\bm y)\bm \phi(\bm y) ds_{\bm y} ,\quad x\in\Gamma,\\
\bm V_{4,22}^w[\psi](x)=& \int_\Gamma E_{22}(\bm x,\bm y) \frac{\psi(\bm y)}{w(\bm y)}ds_{\bm y} ,\quad \bm x\in\Gamma.  
\end{align*}
The regularized formulations for $\bm V_{4,11}^w, \bm V_{4,12}^w$ and $\bm V_{4,21}^w$ are concluded in the following theorem. Several new notations will be used. Denote $\mathbb{J}_{\bm \nu_{\bm x} ,\bm \nu_{\bm y}}:=\bm\nu_{\bm y} \bm \nu_{\bm x}^{\top}-\bm\nu_{\bm x} \bm \nu_{\bm y}^{\top}$. The constants $C_i, \,i=1,2,3$ are defined by
\begin{align*}
C_1&=\frac{i\omega\eta\gamma(k_p^{2}+k_1^{2})-k_1^{2}(k_1^{2}-q)(\lambda+2\mu)}{k_1^{2}-k_2^{2}}, \\
C_2&=\frac{i\omega\eta\gamma(k_p^{2}+k_2^{2})-k_2^{2}(k_2^{2}-q)(\lambda+2\mu)}{k_1^{2}-k_2^{2}}, \\
C_3&=\frac{2\mu}{k_1^{2}-k_2^{2}}(\frac{i\omega\eta\gamma}{\lambda+2\mu}-k_2^{2}+q). 
\end{align*}
The weighted tangential operator $D_s^w$ is given by
\ben
D_s^w[v]= w^2D_s[v] + \frac{v}{2}D_s[w^2],
\enn
with $D_s$ denoting the tangential derivative. To simplify the proof of the regularized formulations, the G\"unter derivative $\bm M(\pa, \bm \nu)$ \cite{BSY20,HW08} is introduced in the form of
\ben 
\bm M(\pa,\bm \nu) \bm u= \pa_{\bm \nu} \bm u-\bm \nu (\nabla \cdot \bm u) + \bm \nu \times \mbox{curl}  \bm u.
\enn
Moreover, the G\"unter derivative and the tangential derivative satisfy the following relation:
\ben
\bm M(\pa, \bm \nu) \bm u = \mathbb{A} D_s [\bm u] , \quad \mathbb{A} =\begin{pmatrix}  0 & -1 \\ 1 & 0 \end{pmatrix}.
\enn

\begin{theorem}
\label{regV3}
The integral operators $\bm V_{4,11}^w, \bm V_{4,12}^w$ and $\bm V_{4,21}^w$ can be re-expressed as
\begin{align}
\label{RegV311}
& \bm V_{4,11}^w = \bm V_{4,11}^{w,1} + D_s\bm V_{4,11}^{w,2}D_s^w + \bm V_{4,11}^{w,3}D_s^w+ D_s\bm V_{4,11}^{w,4} ,\\
\label{RegV312}
& \bm V_{4,12}^w = \bm V_{4,12}^{w,1} + D_s\bm V_{4,12}^{w,2}, \\
\label{RegV313}
& \bm V_{4,21}^w= \bm V_{4,21}^{w,1}+ \bm V_{4,21}^{w,2}D_s^w,
\end{align}
where $\bm V_{4,11}^{w,1}[\bm\phi]=\bm V_{4,11}^{1}[w\bm\phi]$, $\bm V_{4,11}^{w,2}[\bm\phi]=\bm V_{4,11}^{2}[w^{-1}\bm\phi]$, $\bm V_{4,11}^{w,3}[\bm\phi]=\bm V_{4,11}^{3}[w^{-1}\bm\phi]$, $\bm V_{4,11}^{w,4}[\bm\phi]=\bm V_{4,11}^{4}[w\bm\phi]$, $\bm V_{4,12}^{w,1}[\psi]=\bm V_{4,12}^{1}[w^{-1}\psi]$, $\bm V_{4,12}^{w,2}[\psi]=\bm V_{4,12}^{2}[w^{-1}\psi]$, $\bm V_{4,21}^{w,1}[\bm\phi]=\bm V_{4,21}^{1}[w\bm\phi]$ and $\bm V_{4,21}^{w,2}[\bm\phi]=\bm V_{4,21}^{2}[w^{-1}\bm\phi]$, and the integral operators
\begin{align*}
\bm V_{4,11}^{1}[\bm\phi]= & -\rho \omega^2 \int_\Gamma \gamma_{k_s}(\bm x,\bm y)(\bm\nu_{\bm x} \bm \nu_{\bm y}^{\top}-\bm\nu_{\bm x}^{\top}\bm \nu_{\bm y} \mathbb{I}_2-\mathbb{J}_{\bm \nu_{\bm x} ,\bm \nu_{\bm y}})\bm\phi(\bm y)d s_{\bm y}\\
& -\int_{\Gamma}[C_1 \gamma_{k_1}(\bm x,\bm y)-C_2 \gamma_{k_2}(\bm x,\bm y)]\bm\nu_{\bm x} \bm \nu_{\bm y}^{\top} \bm\phi(\bm y)d s_{\bm y},\\
\bm V_{4,11}^{2}[\bm\phi] = & \int_\Gamma \left[4\mu\gamma_{k_s}(\bm x,\bm y) \mathbb{I}_2+ 4\mu^2 \mathbb{A} \mathbb{E}_{11}(\bm x,\bm y) \mathbb{A} \right]\bm\phi(\bm y)d s_{\bm y},\\
\bm V_{4,11}^{3}[\bm\phi] = &  \int_\Gamma \bm\nu_{\bm x} \nabla_{\bm x}^{\top} [-2\mu \gamma_{k_s}(\bm x,\bm y)+(2\mu-C_3)\gamma_{k_1}(\bm x,\bm y)+C_3\gamma_{k_2}(\bm x,\bm y)] \mathbb{A} \bm\phi(\bm y)d s_{\bm y},\\
\bm V_{4,11}^{4}[\bm\phi] = &  \int_\Gamma \mathbb{A}\nabla_{\bm y} [-2\mu \gamma_{k_s}(\bm x,\bm y)+(2\mu-C_3)\gamma_{k_1}(\bm x,\bm y)+C_3\gamma_{k_2}(\bm x,\bm y)] \bm\nu_{\bm y}^\top \bm\phi(\bm y)d s_{\bm y},\\
\bm V_{4,12}^{1}[\psi]= & -\frac{\gamma k_p^2}{k_1^2 - k_2^2}\int_\Gamma \bm\nu_{\bm x}\left[\gamma_{k_1}(\bm x,\bm y)- \gamma_{k_2}(\bm x,\bm y)\right] \psi(\bm y)ds_{\bm y},\\
\bm V_{4,12}^{2}[\psi]= & \frac{2\mu\gamma}{(\lambda  + 2\mu )(k_1^2 - k_2^2)} \int_\Gamma \mathbb{A}\nabla_{\bm x}\left[\gamma_{k_1}(\bm x,\bm y) - \gamma _{k_2}(\bm x,\bm y)\right]\psi(\bm y)ds_{\bm y},\\
\bm V_{4,21}^{1}[\bm\phi]= & -\frac{i\omega\eta k_p^2}{k_1^2 - k_2^2}\int_\Gamma \left[\gamma_{k_1}(\bm x,\bm y)- \gamma_{k_2}(\bm x,\bm y)\right]\bm\nu_{\bm y}^\top \bm\phi(\bm y)ds_{\bm y},\\
\bm V_{4,21}^{2}[\bm\phi]= & -\frac{2i\omega\eta\mu}{(\lambda  + 2\mu )(k_1^2 - k_2^2)}\int_\Gamma \nabla_{\bm x}^\top\left[\gamma_{k_1}(\bm x,\bm y) - \gamma _{k_2}(\bm x,\bm y)\right] \mathbb{A}\bm\phi(\bm y)ds_{\bm y},
\end{align*}
are all at most weakly-singular.
\end{theorem}
\begin{proof}
Since the kernel of the operator $\bm V_{4,11}^w$ is similar as the operator $N_1$ in \cite{ZXY21}, following the proof of \cite[Lemma 3.3]{ZXY21} and carefully treating the tangential derivative $D_s(w\bm\phi)$ yield the regularized formulation (\ref{RegV311}). Note that
\ben
\bm T(\partial ,\bm\nu )\nabla  = (\lambda  + 2\mu )\bm \nu \Delta  + 2\mu \bm M(\partial ,\bm\nu )\nabla ,
\enn
since
\ben
{\partial _{\bm\nu} }\nabla  - \bm M(\partial ,\bm\nu )\nabla  = \bm \nu \Delta .
\enn
Then we have
\begin{align*}
\bm V_{4,12}^w[\psi](x)=& -\frac{\gamma k_p^2}{k_1^2 - k_2^2}\int_\Gamma \bm\nu_{\bm x}\left[\gamma_{k_1}(\bm x,\bm y)- \gamma_{k_2}(\bm x,\bm y)\right] \frac{\psi(\bm y)}{w(\bm y)}ds_{\bm y}\\ 
& +\frac{2\mu\gamma}{(\lambda  + 2\mu )(k_1^2 - k_2^2)}D_{s_{\bm x}}\int_\Gamma \mathbb{A}\nabla_{\bm x}\left[\gamma_{k_1}(\bm x,\bm y) - \gamma _{k_2}(\bm x,\bm y)\right]\frac{\psi(\bm y)}{w(\bm y)}ds_{\bm y},
\end{align*}
and
\begin{align*}
\bm V_{4,21}^w[\bm\phi](x)=& -\frac{i\omega\eta k_p^2}{k_1^2 - k_2^2}\int_\Gamma \left[\gamma_{k_1}(\bm x,\bm y)- \gamma_{k_2}(\bm x,\bm y)\right]\bm\nu_{\bm y}^\top \bm\phi(\bm y)w(\bm y)ds_{\bm y}\\
& -\frac{2i\omega\eta\mu}{(\lambda  + 2\mu )(k_1^2 - k_2^2)}\int_\Gamma \nabla_{\bm x}^\top\left[\gamma_{k_1}(\bm x,\bm y) - \gamma _{k_2}(\bm x,\bm y)\right] \mathbb{A}\frac{D_s^w(\bm\phi)(\bm y)}{w(\bm y)}ds_{\bm y},
\end{align*}
which completes the proof.
\end{proof}

\section{Numerical experiments}
\label{sec:4}

\subsection{Numerical implementations}
\label{sec:4.1}

We utilize the Nystr\"om-type method~\cite{BL12,BXY21} to numerically discretize the integral operators \jp{$\bm{V}_i^{w} , i = 1,2,3,4 $} on the basis of the Chebyshev quadrature points $\theta_n=\frac{\pi (2n+1)}{2N},\,\, n=0,1,...,N-1$, and truncated Chebyshev expansions of the smooth densities. Without loss of generality, we assume a smooth parameterization $\jp{\bm x} (t)=(x_1 (t),x_2(t)),\,\,t\in [-1,1]$ of $\Gamma$ satisfying $|\jp{\bm x'}(t)|=|d\jp{\bm x}(t)/dt|>0$, and choose
the weight $w$ so that $w(\jp{\bm x}(t))=\sqrt{1-t^2}$. Hence, the operator $\jp{\bm{V}_1^w}$ will be given by \ben
\jp{\bm{V}_1^w}[\bm \phi](t)=\int_{-1}^1
\jp{\mathbb{E}^\top(\bm x(t),\bm x(\tau))}\frac{\bm \phi(\bm x(\tau))}{\sqrt{1-\tau^2}}|\bm x'(\tau)|d\tau. \enn
Introducing the change of variables $t=\cos\theta$ and
$\tau=\cos\theta'$, we
obtain the periodic weighted operator
\be
\jp{\bm{V}_1^w}[\bm \phi](\theta)=\int_0^\pi
\mathbb{E}^\top(\bm x(\cos\theta),\bm x(\cos\theta'))\widetilde{\bm \phi}(\theta')|\bm x'(\cos\theta')|d\theta',
\quad \widetilde{\bm \phi}(\theta')=\bm \phi(\bm x(\cos\theta')).  \en Using the series expansions of the Hankel functions  \cite{CK13}, we can
obtain the decomposition
\be \label{anal} \jp{\mathbb{E}}(\bm x(t),\bm x(\tau))=\mathbb{E}^{1}(t,\tau)\log|t-\tau|+\mathbb{E}^{2}(t,\tau),  \en
where $\mathbb{E}^1(t,\tau)$ and $\mathbb{E}^2(t,\tau)$ are analytic. The detailed expressions of $\mathbb{E}^1$ and $\mathbb{E}^2$ are presented in the appendix.

It should be noted that, since the wave numbers 
$k_1,k_2$
  are complex, the terms 
$\mathbb{E}^{1}$ and $\mathbb{E}^{2}$
  increase exponentially with frequency 
$\omega$. An alternative regularization technique involves selecting only a few terms from the power series of 
$\mathbb{E}^1(t,\tau)$, which results in new functions 
$\mathbb{E}^1(t,\tau), \mathbb{E}^2(t,\tau)$ increasing polynomially with frequency. However, this approach means the overall scheme will not converge spectrally, but instead with a fixed order, determined by the number of terms selected in the regularization. For this work, we will limit ourselves to low- and medium-frequency regimes and therefore we use the traditional regularization scheme.

Using the Chebyshev points
$\left\{\theta_n=\frac{\pi(2n+1)}{2N}\right\}, n=0,1,\cdots,N-1$, gives
rise to a spectrally convergent cosine representation for smooth,
$2\pi$-periodic and even function $ \widetilde{\bm \phi}$ as \be
\label{periodicexp}
\jp{
\widetilde{\bm \phi}(\theta)=\sum_{n=0}^{N-1}\begin{pmatrix}
a^1_n\\ a^2_n
\end{pmatrix} e_n(\theta), \quad \begin{pmatrix}
a^1_n\\ a^2_n
\end{pmatrix}=\frac{2-\delta_{0n}}{N}\sum_{j=0}^{N-1}\widetilde{\bm \phi}(\theta_j)\cos(n\theta_j),}
\en
where $e_n(\theta) = \cos(n\theta)$. \jp{From \cite{BH07,YS88}, we have that}
\be
\label{symmpro}
-\frac{1}{2\pi}\int_0^\pi \log|\cos\theta-\cos\theta'|e_n(\theta')d\theta' =\lambda_ne_n(\theta), \quad \lambda_n=\begin{cases}
\frac{\log2}{2}, & n=0, \\
\frac{1}{2n}, & n\ge 1.
\end{cases}
\en
Applying equation (\ref{symmpro}) to each term of expansion (\ref{periodicexp}) we can obtain the well-known spectral quadrature rule for the logarithmic kernel
\be
\label{sqr}
\int_0^\pi \log|\cos\theta-\cos\theta'|\widetilde{\bm \phi}(\theta')d\theta' \jp{\approx} \frac{\pi}{N}\sum_{j=0}^{N-1}\widetilde{\bm \phi}(\theta_j)R_j^{(N)}(\theta),
\en
where
\be
\label{RjN}
R_j^{(N)}(\theta)=-2\sum_{m=0}^{N-1}(2-\delta_{0m})
\lambda_m\cos(m\theta_j)\cos(m\theta).  \en Together with the
trapezoidal integration for smooth periodic functions we  obtain a
spectrally quadrature approximation of the operator $\bm{V}_1^{w}$ as \be
\label{sqaS}
\bm{V}_1^{w}[\bm \phi](\theta)\approx  \frac{\pi}{N}\sum_{j=0}^{N-1} |\bm x'(\cos\theta_j)| \left[\mathbb{E}^1 (\cos\theta,\cos\theta_j)R_j^{(N)}(\theta) + \mathbb{E}^2 (\cos\theta,\cos\theta_j)\right]\widetilde{\bm \phi}(\theta_j).
\en
Then we can evaluate $\bm{V}_1^w[\bm \phi]$ in the sets of quadrature points $\{\theta_n, n=0,\cdots,N-1\}$ by means of a matrix-vector multiplication involving the \jp{block}-matrix ${\bf V}_1^w$ whose elements are defined by
\ben
[{\bf V}_1^w]_{ij}=\frac{\pi}{N}|\bm x'(\cos\theta_j)| \left[\mathbb{E}^1 (\cos\theta_i,\cos\theta_j)R_j^{(N)}(\theta_i) + \mathbb{E}^2 (\cos\theta_i,\cos\theta_j)\right],
\enn
in which the quantities $R_j^{(N)}(\theta_i)$ can be evaluated through
\ben
R_j^{(N)}(\theta_i)=R^{(N)}(|i-j|)+R^{(N)}(i+j+1),
\enn
where
\ben
R^{(N)}(l)=-\sum_{m=0}^{N-1} (2-\delta_{0m})\lambda_m\cos\left(\frac{lm\pi}{N}\right), \quad l=0,\cdots,2N-1.
\enn

In order to discretize the operators $\bm{V}_i^w,i=2,3,4$, we use the regualarized technique presented in Section~\ref{sec:3.3}. From the regualarized formulations given in Theorem~\ref{regV3} and in the appendix, the numerical evaluations of the operators $\bm{V}_i^w, i=2,3,4$ reduce to calculate the operators with the following forms
\be \bm{S}_1, \quad D_s \bm{S}_2 D_s^w, \quad \bm{S}_3 D_s^w, \quad D_s \bm{S}_4, \en
where $\bm{S}_i,\,i=1,2,3,4$ are integral operators with kernel of logarithmic type (which can be treated like $\bm{V}_1^w$), and
\begin{align} D_s[\alpha](\theta)=-\frac{1}{|\bm x'(\cos\theta)|}\frac{1}{\sin\theta} \frac{d\alpha(\theta)}{d\theta},
\end{align}
\begin{align}
  D_s^w[\alpha](\theta)=-\frac{\sin\theta}{|\bm x'(\cos\theta)|} \frac{d\alpha(\theta)}{d\theta} - \frac{\alpha(\theta)\cos\theta}{|\bm x'(\cos\theta)|}. \end{align}
Note that $\alpha(\theta)$ is an even function on $[0,\pi]$ which can be extended to be an even $2\pi$-periodic function on $[-\pi,\pi]$ and then whose derivatives respect to $\theta$ can be evaluated by the FFT.

\subsection{Numerical examples}
\label{sec:4.2}

Now we present several numerical examples to demonstrate the accuracy and efficiency of the proposed regularized BIE methods for solving the thermo-elastic scattering by open-arcs. We fix the parameters,  $\rho=1, \mu=1, \lambda=2,
\kappa=1, \eta=0.2$ and $\gamma=0.1$.  The boundary data for the different boundary condition will be given by 
\ben  \bm G_i=- \mathbb{B}_i(\pa,\nu) \bm U^{inc}, \,\, i=1,2,3,4,\enn
where
$\bm U^{inc}=(\bm u^{inc },0)$, with $\bm u^{inc}=\bm d_{inc}^{\bot }\mbox{exp}(i k_s \bm x \cdot \bm d_{inc})$ a plane incident shear wave with direction $\bm d_{inc}=(\cos \theta_{inc},\sin \theta_{inc})$, $\bm d_{inc}^{\bot}=(-\sin \theta_{inc}, \cos \theta_{inc})$ and $\theta_{inc}$ denoting the incident angle.

In what follows we will denote by $\mathrm{Solver}_i(\bm{V}_i^w)$ whenever we are using the  unregularized solvers to approximate the i-th BIE in \eqref{weighted} (i.e., solving the i-th boundary value problem), for $i=1,2,3,4$. Similarly we denote by  $\mathrm{Solver}_i(\bm{V}_{j_i}^w\bm{V}_{j_i-1}^w)$ when we use the regularized solvers, specified on the third column of Table~\ref{RsolrepHIE}, with $j_1=j_2=2,j_3=j_4=4$. All the formulations were implemented in Matlab using the expressions of Section \ref{sec:4.1}, while the resulting linear system were approximated using GMRES. 


{\bf Example 1.} In our first example, we test the accuracy of the proposed method. Consider the thermoelastic scattering by a spiral-shaped arc which is parametrized by $\bm x(t)=e^{t}(\cos 5t,\sin 5t), t\in[-1,1]$.

The corresponding approximated solution are denoted by $\bm {U}_N$, where $N$ is the number of points of the discretization. We measure the relative errors of the near-field as,  
\ben 
\bm \epsilon_{\infty}:=\frac{ \max_{\bm x \in \bm C} | {\bm U}_{N}(\bm x) - {\bm U}_{N_K} (\bm x)|}{\max_{\bm x \in \bm C} |\bm U_{N_K} (\bm x)|},
\enn
for a (big) fixed value of $N_K=1600$, and different values of $N$. The points $\bm x $ are selected to be on a circle $\bm C =\{ \|\bm x\|=4, \bm x \in \mathbb{R}^2\}$ enclosing the open arc. The results, for $\theta_{inc} = 0$, $\omega=10,50$, and the different formulations (unregularized and regularized), are presented in  Table~\ref{Table2}  for the pure Dirichlet and Neumann cases, and Table~\ref{Table3} for the mixed boundary conditions. In all cases we observe fast convergence, achieving small errors. Particularly for the problems with traction boundary conditions ($i=2,4$), the regularized BIE approach can efficiently improve the accuracy.

Complementary to this example  we also present the near fields obtained from the regularized formulation for the scattering by the flat strip $[-1,1]\times \{0\}$ and the "fish-shaped arc"
\ben \bm x(t)=(1+0.7\cos1.6 \pi t)(\cos0.8\pi t,\sin0.8\pi t),\quad t\in[-1,1], \enn
in Figures~\ref{FigESol12}  and \ref{FigESol34}. 


\begin{table}[htbp]\small
    \caption{Near-field errors for thermoelastic scattering by a spiral-shaped arc. GMRES tol: $10^{-12}$.}
    \centering
    \begin{tabular}{ccccccc}
    \hline
    $\omega$ & $N$ & $\mbox{Solver}_1(\bm{V}_1^w)$ & $\mbox{Solver}_2(\bm{V}_2^w)$ & $\mbox{Solver}_1(\bm{V}_2^w \bm{V}_1^w)$ & $\mbox{Solver}_2(\bm{V}_2^w \bm{V}_1^w)$  \\
    \hline
     10 & 100  & $3.01\times10^{-2}$  & $2.42\times10^{-1}$   & $3.01\times10^{-2}$ & $2.41\times10^{-1}$  \\
        & 150  & $1.81\times10^{-5}$  & $1.57\times10^{-3}$   & $1.81\times10^{-5}$ & $1.57\times10^{-3}$  \\
        & 300  & $4.75\times10^{-14}$  & $4.31\times10^{-9}$  & $1.41\times10^{-11}$ & $1.55\times10^{-12}$  \\
    \hline
     50 & 600  & $3.32\times10^{-3}$  & $2.23\times10^{-1}$   & $3.32\times10^{-3}$ & $1.81\times10^{-1}$  \\
        & 800  & $4.18\times10^{-10}$  & $1.64\times10^{-6}$  & $3.13\times10^{-10}$ & $5.89\times10^{-7}$ \\
        & 1000  & $3.30\times10^{-10}$ & $2.26\times10^{-8}$  & $1.18\times10^{-10}$ & $6.53\times10^{-10}$ \\
    \hline
    \end{tabular}
    \label{Table2}
    \end{table}

    \begin{table}[htbp]\small
      \caption{Near-field errors for thermoelastic scattering by a spiral-shaped arc. GMRES tol: $10^{-12}$.}
      \centering
      \begin{tabular}{ccccccc}
      \hline
      $\omega$ & $N$ & $\mbox{Solver}_3(\bm{V}_3^w)$ & $\mbox{Solver}_4(\bm{V}_4^w)$  & $\mbox{Solver}_3(\bm{V}_4^w \bm{V}_3^w)$ & $\mbox{Solver}_4(\bm{V}_4^w \bm{V}_3^w)$\\
      \hline
       10 & 100    & $2.99\times10^{-2}$ & $2.42\times10^{-1}$ & $2.99\times10^{-2}$ & $2.42\times10^{-1}$  \\
          & 150    & $1.81\times10^{-5}$ & $1.57\times10^{-3}$ & $1.81\times10^{-5}$ & $1.57\times10^{-3}$  \\
          & 300    & $8.86\times10^{-11}$ & $1.23\times10^{-9}$ & $3.43\times10^{-11}$ & $1.21\times10^{-13}$  \\
      \hline
       50 & 600    & $3.32\times10^{-3}$ & $1.82\times10^{-1}$ & $3.32\times10^{-3}$ & $1.82\times10^{-1}$  \\
          & 800    & $1.01\times10^{-8}$ & $5.91\times10^{-7}$ & $1.22\times10^{-9}$ & $5.91\times10^{-7}$ \\
          & 1000   & $1.07\times10^{-8}$ & $2.33\times10^{-9}$ & $9.10\times10^{-10}$ & $7.85\times10^{-10}$ \\
      \hline
      \end{tabular}
      \label{Table3}
      \end{table}
      
\begin{figure}[htbp]
  \centering
  \begin{tabular}{ccc}
  \includegraphics[scale=0.15]{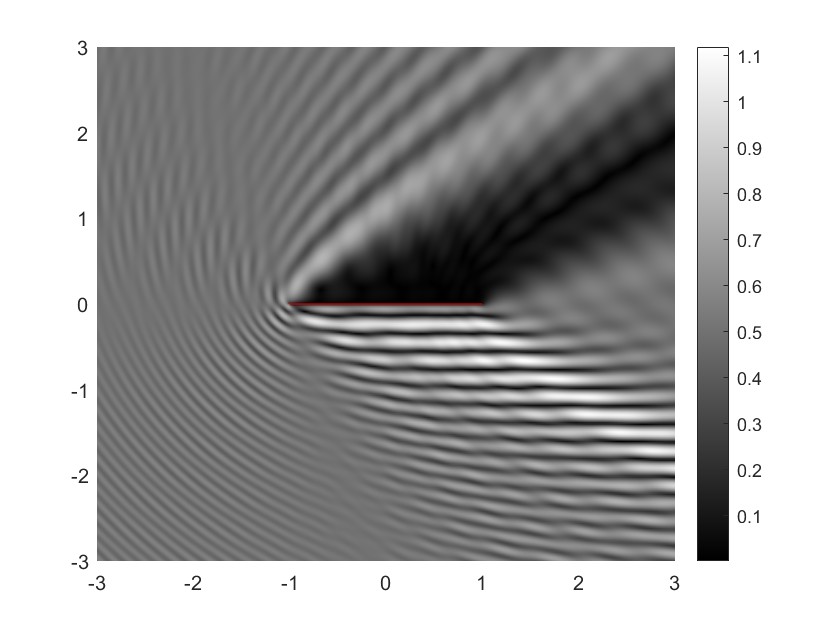} &
  \includegraphics[scale=0.15]{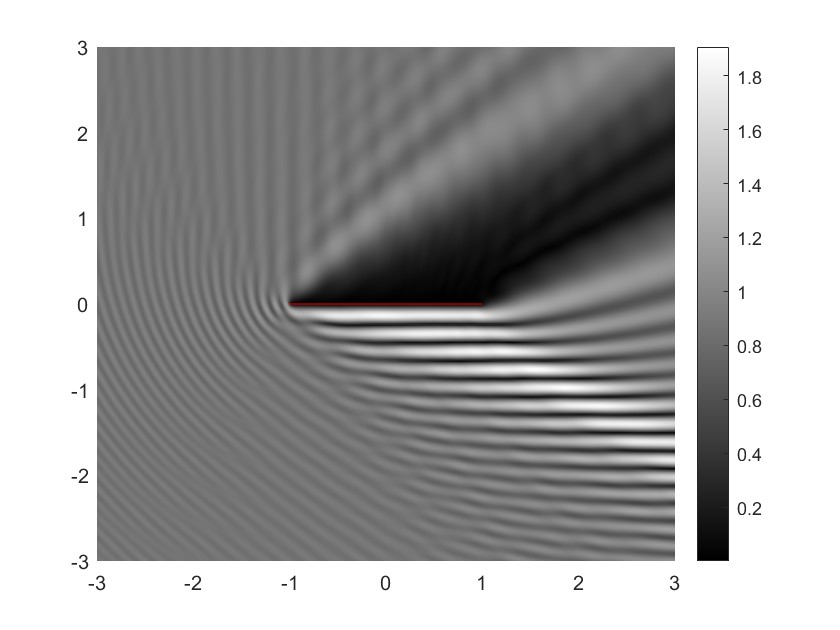} &
   \includegraphics[scale=0.15]{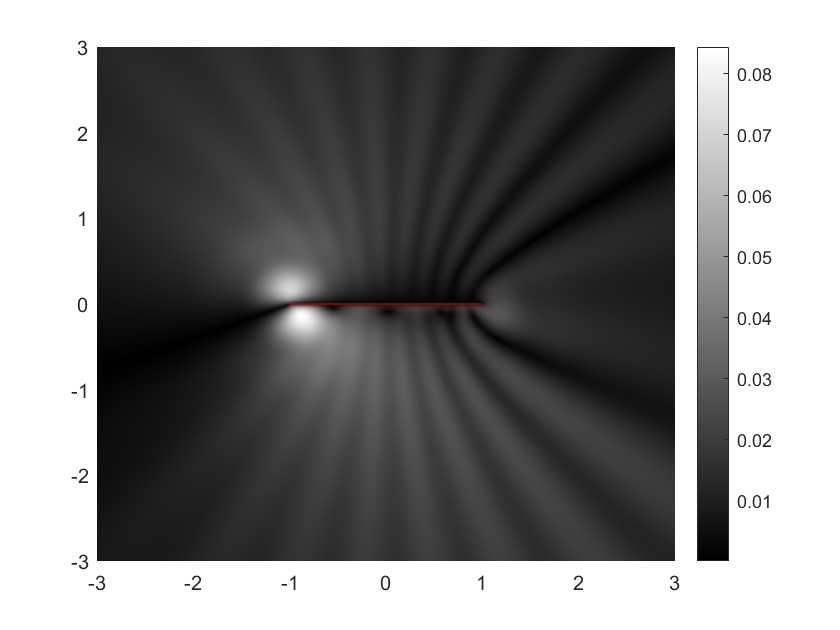}\\
  (a) $|u_1|, \theta_{inc}=\pi/6$ & (b) $|u_2|, \theta_{inc}=\pi/6$ & (c) $|p|, \theta_{inc}=\pi/6$\\
  \includegraphics[scale=0.15]{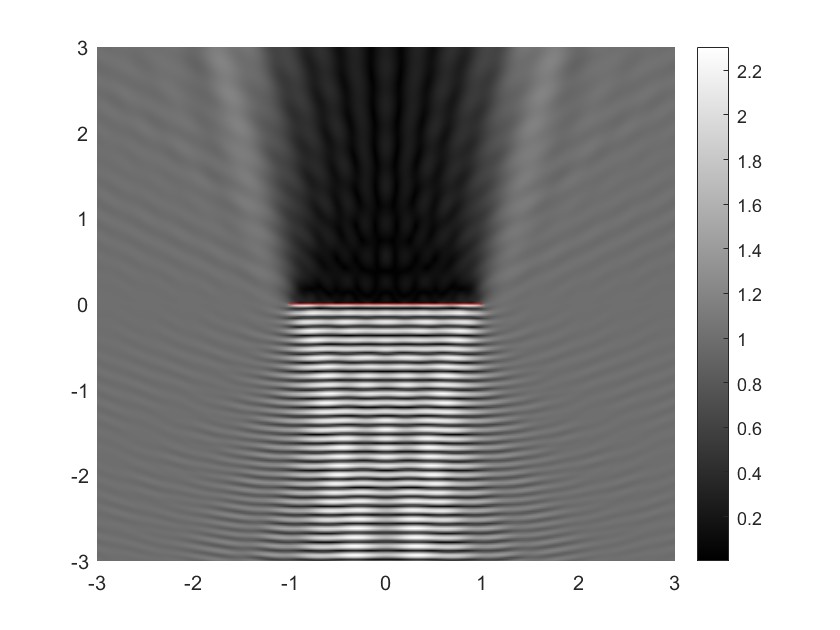} &
  \includegraphics[scale=0.15]{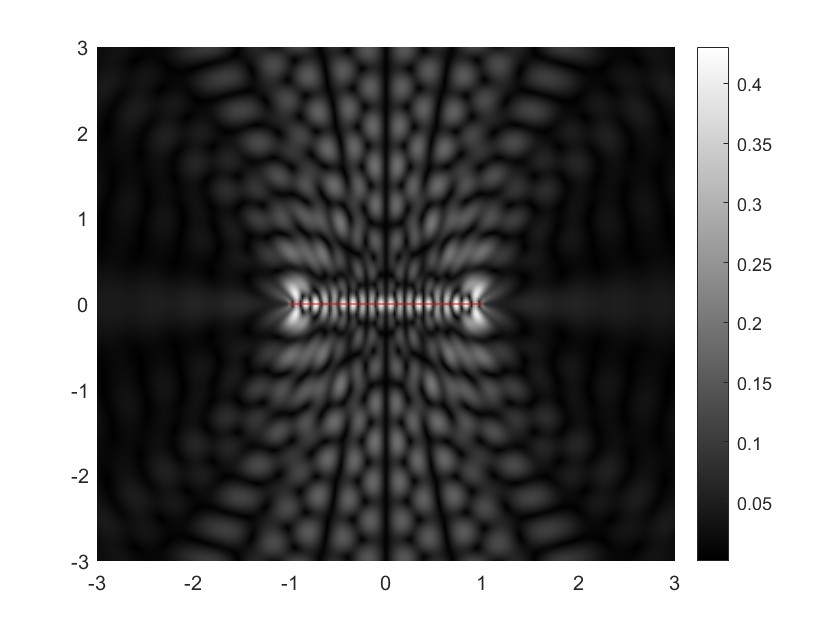} &
  \includegraphics[scale=0.15]{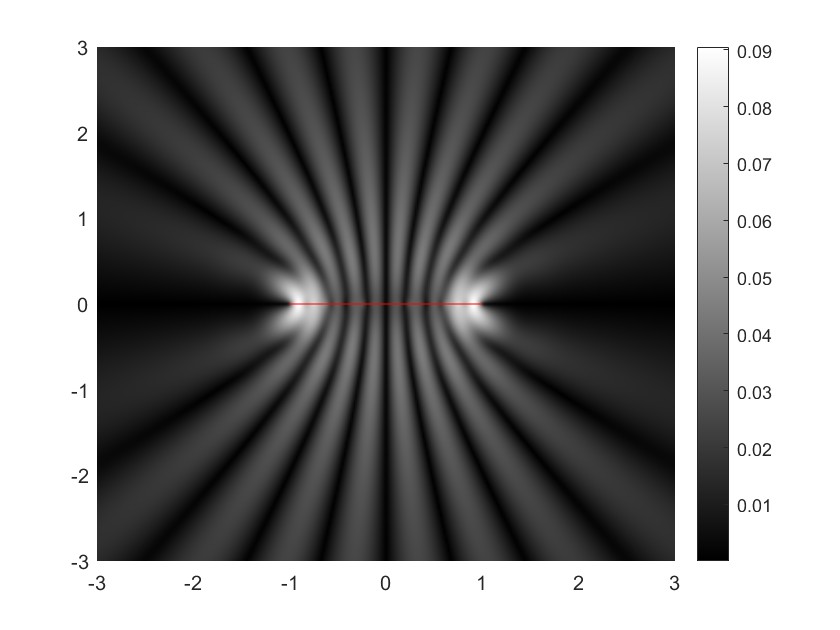} \\
  (d) $|u_1|, \theta_{inc}=\pi/2$ &
    (e) $|u_2|, \theta_{inc}=\pi/2$ &
    (f) $|p|, \theta_{inc}=\pi/2$ \\
  \end{tabular}
  \caption{Thermoelastic scattering by a flat strip with 1-st (a,b,c) and 2-nd (d,e,f) type boundary condition. ($\omega=30$)}
  \label{FigESol12}
  \end{figure}

\begin{figure}[htbp]
  \centering
  \begin{tabular}{ccc}
  \includegraphics[scale=0.15]{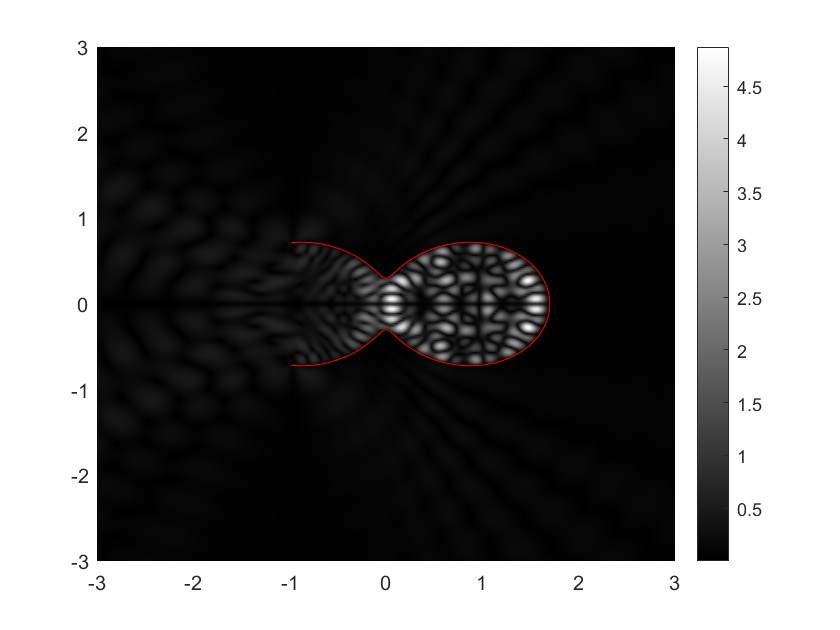} &
  \includegraphics[scale=0.15]{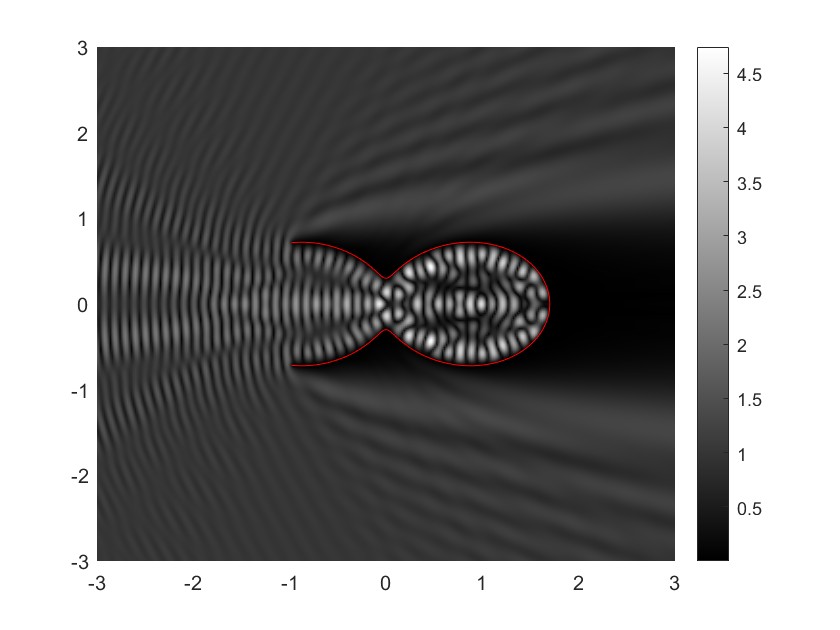} &
  \includegraphics[scale=0.15]{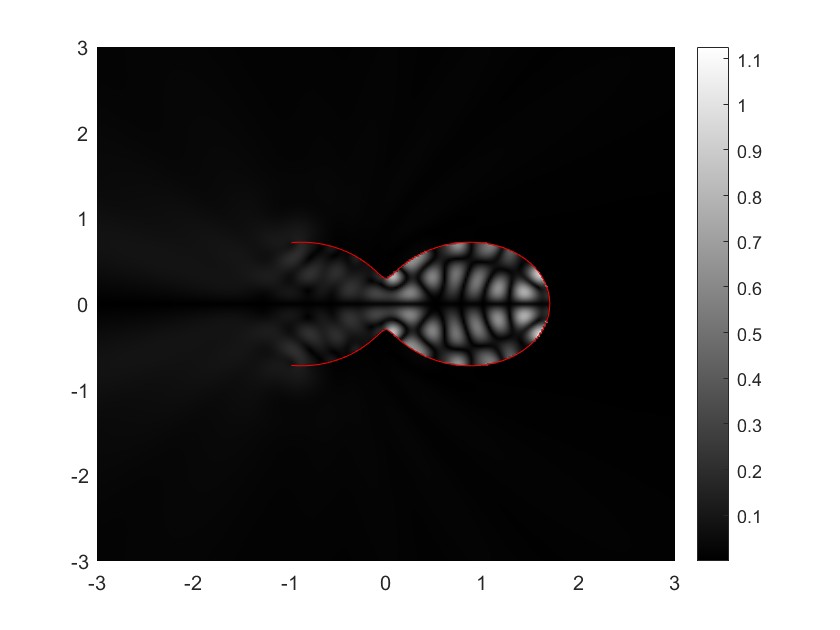} \\
  (a) $|u_1|, \theta_{inc}=0$ & (b) $|u_2|, \theta_{inc}=0$  & (c) $|p|, \theta_{inc}=0$\\
  \includegraphics[scale=0.15]{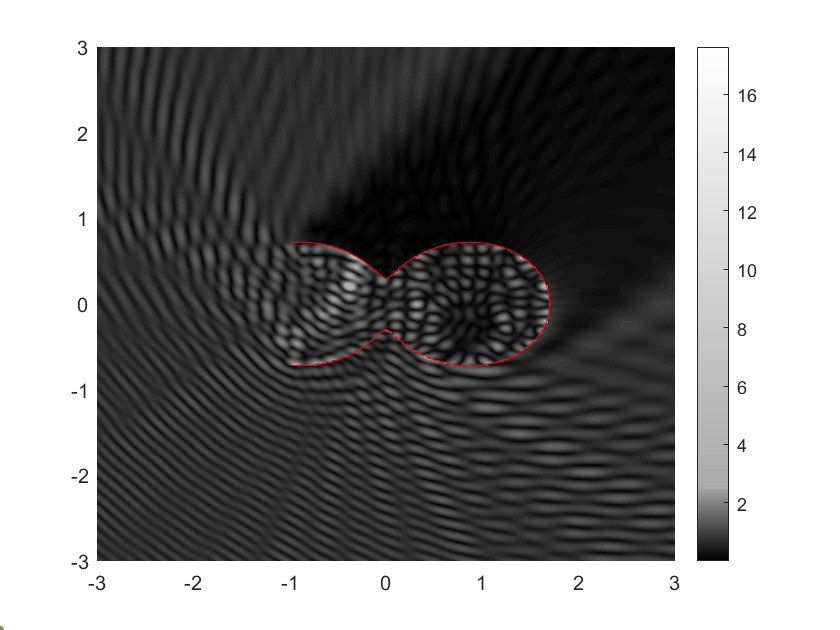} &
  \includegraphics[scale=0.15]{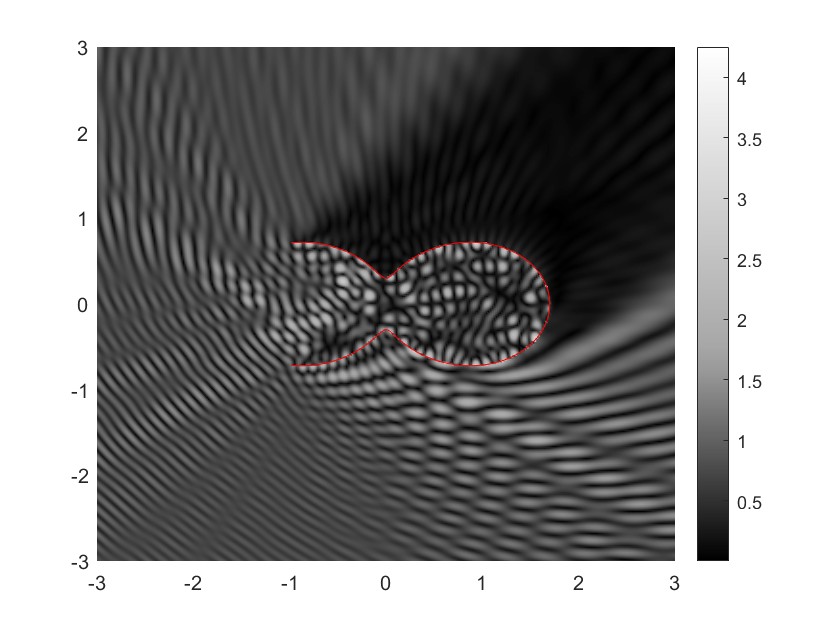} &
  \includegraphics[scale=0.15]{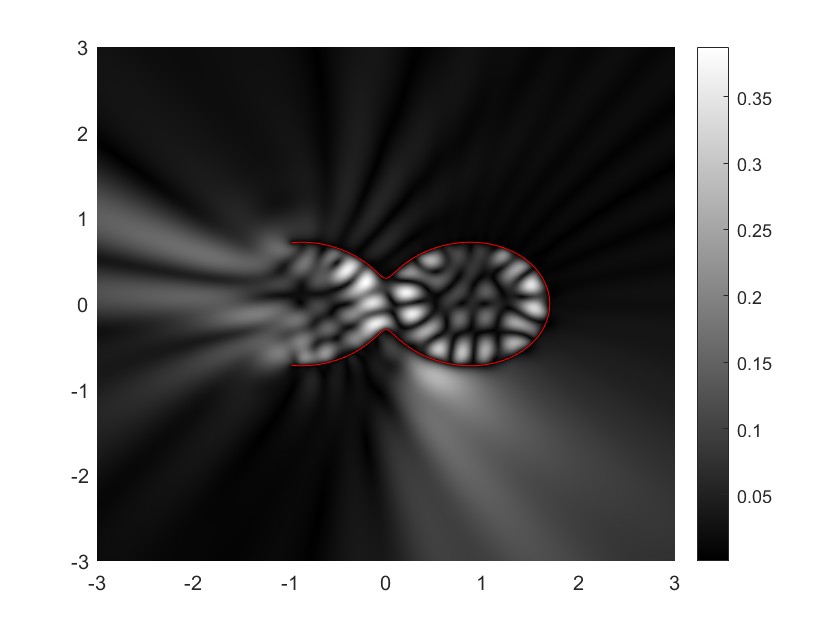} \\
 (d) $|u_1|, \theta_{inc}=\pi/4$ &
    (e) $|u_2|, \theta_{inc}=\pi/4$ &
    (f) $|p|, \theta_{inc}=\pi/4$ \\
  \end{tabular}
  \caption{Thermoelastic scattering by a "fish-shaped arc" with 3-rd (a,b,c) and 4-th (d,e,f) type boundary condition. ($\omega=30$)}
  \label{FigESol34}
  \end{figure}
  
{\bf Example 2.} For the second example we focus on showing the efficiency of the proposed regularized BIE solvers. We consider the flat strip and the "fish-shaped arc" geometries that were detailed at the end of the previous example.

Tables~\ref{Table4}-\ref{Table7} present the computation time and the number of GMRES iterations required to achieve a residual tolerance of $10^{-10}$ for the different formulations, with frequencies $\omega=10, 50$ and $\theta_{inc}=\frac{\pi}{4}$. In these tables, $T_p$ refers to the time for "precomputation"---which mainly amounts to the assembly time associated with the numerical evaluation of necessary matrices of the weakly singular integral operators as mentioned in section \ref{sec:4.1},  $N_{it}$ denotes the number of GMRES iterations and $T_{it}$  denotes GMRES time to perform all $N_{it}$ iterations.

We observe that for all the considered cases the  ``regularized'' solvers lead to a significant reduction in GMRES iterations vis-a-vis the numbers of iterations required for the  ``unregularized'' solvers.  Only for the pure Dirichlet case we observe that the total times ($T_p+T_{it}$) is higher for the regularized solver that for the unregularized one. The reason for this seems to be a combination of a relative less significant reduction on the GMRES iterations, and the added assambly required for the construction of the regularized. Therefore, it is suggested to utilize $\mathrm{Solver}_1(\bm V_1^w)$ to numerically solve the Dirichlet problem, which shows to be much faster than $\mathrm{Solver}_1(\bm V_2^w\bm V_1^w)$, in spite of the fact that $\mathrm{Solver}_1(\bm V_1^w)$ requires larger numbers of iterations than $\mathrm{Solver}_1(\bm V_2^w\bm V_1^w)$. For the 2nd, 3rd and 4th type boundary conditions, although the regularized solvers $\mbox{Solver}_2(\bm{V}_2^w \bm{V}_1^w)$, $\mbox{Solver}_3(\bm{V}_4^w \bm{V}_3^w)$ and $\mbox{Solver}_4(\bm{V}_4^w \bm{V}_3^w)$ requires to evaluate an additional operator compared with the unregularized solvers $\mathrm{Solver}_2(\bm V_2^w)$, $\mathrm{Solver}_3(\bm V_3^w)$ and $\mathrm{Solver}_4(\bm V_4^w)$, respectively, the regularized solvers require much fewer iterations than the unregularized solvers which shows that the formers are significantly faster overall. Hence, use of the regularized solvers is recommended for the 2nd, 3rd and 4th type boundary conditions.

\begin{table}[htbp]\small
      \caption{Number of iterations and computing time (seconds) required for the flat-strip thermoelastic scattering problem. GMRES tol: $10^{-10}$.}
      \centering
      \resizebox{\linewidth}{!}{
      \begin{tabular}{m{0.5cm}<{\centering}|m{0.8cm}<{\centering}|m{1cm}<{\centering}m{1cm}<{\centering}m{1cm}<{\centering}|m{1cm}<{\centering}m{1cm}<{\centering}m{1cm}<{\centering}|m{1cm}<{\centering}m{1cm}<{\centering}m{1cm}<{\centering}|m{1cm}<{\centering}m{1cm}<{\centering}m{1cm}<{\centering}}
      \hline
      &  & \multicolumn{3}{c|}{$\mbox{Solver}_1(\bm{V}_1^w)$} &  \multicolumn{3}{c|}{$\mbox{Solver}_1(\bm{V}_2^w \bm{V}_1^w)$} & \multicolumn{3}{c|}{$\mbox{Solver}_2(\bm{V}_2^w)$}    & \multicolumn{3}{c}{$\mbox{Solver}_2(\bm{V}_2^w \bm{V}_1^w)$} \\
      \hline
       $\omega$& N & $T_p$ & $N_{it}$ & $T_{it}$ & $T_p$ & $N_{it}$ & $T_{it}$ & $T_p$ & $N_{it}$ & $T_{it}$ & $T_p$ & $N_{it}$ & $T_{it}$\\
      \hline 
       10 & 150  &0.25  &35 &0.02        &0.5 &19 &0.1        &0.5 &48 &0.2        &0.5  &12  &0.1 \\
          & 300  &0.7  &35 &0.05        &1.2 &19 &0.25       &1.1 &69 &0.6        &1.1  &12  &0.2 \\
       \hline
       50 & 800  &4.1  &83 &0.6        &6.5 &47 &2.3        &5.9 &196 &7.5        &6.0  &33  &1.4 \\
          & 1000  &6.3  &83 &0.9        &9.9 &47 &3.2        &9.7 &219 &12.3        &9.3  &33  &2.2 \\

      \hline
      \end{tabular}
      }
      \label{Table4}
      \end{table}



\begin{table}[htbp]\small
      \caption{Number of iterations and computing time (seconds) required for the fish-shaped arc thermoelastic scattering problem. GMRES tol: $10^{-10}$.}
      \centering
      \resizebox{\linewidth}{!}{
      \begin{tabular}{m{0.5cm}<{\centering}|m{0.8cm}<{\centering}|m{1cm}<{\centering}m{1cm}<{\centering}m{1cm}<{\centering}|m{1cm}<{\centering}m{1cm}<{\centering}m{1cm}<{\centering}|m{1cm}<{\centering}m{1cm}<{\centering}m{1cm}<{\centering}|m{1cm}<{\centering}m{1cm}<{\centering}m{1cm}<{\centering}}
      \hline
       &  & \multicolumn{3}{c|}{$\mbox{Solver}_1(\bm{V}_1^w)$} &  \multicolumn{3}{c|}{$\mbox{Solver}_1(\bm{V}_2^w \bm{V}_1^w)$} & \multicolumn{3}{c|}{$\mbox{Solver}_2(\bm{V}_2^w)$}    & \multicolumn{3}{c}{$\mbox{Solver}_2(\bm{V}_2^w \bm{V}_1^w)$} \\
      \hline
       $\omega$& N & $T_p$ & $N_{it}$ & $T_{it}$ & $T_p$ & $N_{it}$ & $T_{it}$ & $T_p$ & $N_{it}$ & $T_{it}$ & $T_p$ & $N_{it}$ & $T_{it}$\\
      \hline
       10 & 150  &0.2  &123 &0.1        &0.5 &62 &0.25        &0.5 &230 &1.0       &0.5  &70  &0.3 \\
          & 300  &0.9  &136 &0.3        &1.1 &58 &0.4        &1.1 &333 &2.7        &1.1  &59  &0.4 \\
          \hline
       50 & 800  &3.7  &313 &3.5        &5.5 &171 &5.0        &5.3 &727 &25.7        &5.3  &174  &5.0 \\
          & 1000  &5.8  &313 &4.9        &8.7 &171 &7.4        &7.7 &827 &40.3        &8.0  &173  &7.0 \\

      \hline
      \end{tabular}
      }
      \label{Table5}
      \end{table}

        


  \begin{table}[htbp]\small
        \caption{Number of iterations and computing time (seconds) required for the flat-strip thermoelastic scattering problem. GMRES tol: $10^{-10}$. (Data marked with $*$ indicates that GMRES did not converge to the specified tolerance because the method stagnated.)}
        \centering
        \resizebox{\linewidth}{!}{
        \begin{tabular}{m{0.5cm}<{\centering}|m{0.8cm}<{\centering}|m{1cm}<{\centering}m{1cm}<{\centering}m{1cm}<{\centering}|m{1cm}<{\centering}m{1cm}<{\centering}m{1cm}<{\centering}|m{1cm}<{\centering}m{1cm}<{\centering}m{1cm}<{\centering}|m{1cm}<{\centering}m{1cm}<{\centering}m{1cm}<{\centering}}
        \hline
        &  & \multicolumn{3}{c|}{$\mbox{Solver}_3(\bm{V}_3^w)$} &  \multicolumn{3}{c|}{$\mbox{Solver}_3(\bm{V}_4^w \bm{V}_3^w)$} &\multicolumn{3}{c|}{$\mbox{Solver}_4(\bm{V}_4^w)$}    & \multicolumn{3}{c}{$\mbox{Solver}_4(\bm{V}_4^w  \bm{V}_3^w)$} \\
        
        \hline
         $\omega$& N & $T_p$ & $N_{it}$ & $T_{it}$ & $T_p$ & $N_{it}$ & $T_{it}$ & $T_p$ & $N_{it}$ & $T_{it}$ & $T_p$ & $N_{it}$ & $T_{it}$\\
        \hline 
         10 & 150       &0.2 &176 &0.4        &0.6  &19  &0.1  &0.4  &210 &0.7       &0.5 &15 &0.1   \\
            & 300       &0.6 &311 &1.4       &1.1  &19  &0.2  &0.9  &310 &2.0        &1.1 &15 &0.2   \\
            \hline
         50 & 800       &4.1 &864 &23.0        &5.9  &47  &1.4  &5.1  &812 &27.3        &5.8 &36 &1.1   \\
            & 1000         &6.3 & 1024* &37.7        &8.9  &47  &2.0  &7.8  &911 &42.1        &9.5 &36 &1.6 \\

        \hline
        \end{tabular}
        }
        \label{Table6}
        \end{table}




      \begin{table}[htbp]\small
        \caption{Number of iterations and computing time (seconds) required for the fish-shaped arc thermoelastic scattering problem. GMRES tol: $10^{-10}$.}
        \centering
        \resizebox{\linewidth}{!}{
        \begin{tabular}{m{0.5cm}<{\centering}|m{0.8cm}<{\centering}|m{1cm}<{\centering}m{1cm}<{\centering}m{1cm}<{\centering}|m{1cm}<{\centering}m{1cm}<{\centering}m{1cm}<{\centering}|m{1cm}<{\centering}m{1cm}<{\centering}m{1cm}<{\centering}|m{1cm}<{\centering}m{1cm}<{\centering}m{1cm}<{\centering}}
        \hline
        &  & \multicolumn{3}{c|}{$\mbox{Solver}_3(\bm{V}_3^w)$} &  \multicolumn{3}{c|}{$\mbox{Solver}_3(\bm{V}_4^w \bm{V}_3^w)$} &\multicolumn{3}{c|}{$\mbox{Solver}_4(\bm{V}_4^w)$}    & \multicolumn{3}{c}{$\mbox{Solver}_4(\bm{V}_4^w  \bm{V}_3^w)$} \\

        \hline
         $\omega$& N & $T_p$ & $N_{it}$ & $T_{it}$ & $T_p$ & $N_{it}$ & $T_{it}$ & $T_p$ & $N_{it}$ & $T_{it}$ & $T_p$ & $N_{it}$ & $T_{it}$\\
        \hline
         10 & 150  &0.2 &251 &0.6        &0.5  &62  &0.2  &0.3  &348 &1.4       &0.6 &63 &0.3         \\
            & 300  &0.8 &412 &2.3        &1.2  &58  &0.5    &1.0  &640 &6.1       &1.2 &59 &0.4        \\
            \hline
         50 & 800 &3.8 &1058 &32.5        &5.6  &172  &5.0   &4.6  &1305 &58.0        &5.5 &174 &5.2         \\
            & 1000  &5.9 &1258 &53.3        &8.7  &171  &7.1   &7.3  &1486 &89.8        &8.4 &174 &7.1         \\

        \hline
        \end{tabular}
        }
        \label{Table7}
        \end{table}


\section*{Acknowledgments}
This work is partially supported by the National Key R\&D Program of China, Grant No. 2024YFA1016000, the Strategic Priority Research Program of the Chinese Academy of Sciences, Grant No. XDB0640000, NSFC under grants 12171465 and 12288201, and ANID grant Fondecyt Iniciaci\'on N°11230248.

\section*{Appendix}
In this appendix, we first present the regularized formulations for $\bm V_i^w, i=2,3$. Using the notations in Section~\ref{sec:3.3}, it can be derived that the integral operators $\bm V_{2,11}^w$, $\bm V_{2,12}^w$, $\bm V_{2,21}^w$, $\bm V_{2,22}^w$ can be re-expressed as

\begin{align*}
    \bm V_{2,11}^w=&\bm V_{4,11}^w,\\
    \bm V_{2,12}^w=&\bm V_{2,12}^{w,1} + D_s \bm V_{2,12}^{w,2} D_s^w+ D_s \bm V_{2,12}^{w,3}, \\
    \bm V_{2,21}^w=&\bm V_{2,21}^{w,1} + D_s \bm V_{2,21}^{w,2} D_s^w +\bm V_{2,21}^{w,3} D_s^w, \\
   \bm V_{2,22}^w=&\bm V_{2,22}^{w,1} + D_s \bm V_{2,22}^{w,2} D_s^w,
\end{align*}
where $\bm V_{4,11}^w$ is defined in (\ref{RegV311}), and $\bm V_{2,12}^{w,1}[\psi]= \bm V_{2,12}^{1}[w \psi]$,
$\bm V_{2,12}^{w,2}[\psi]= \bm V_{2,12}^{2}[w^{-1}\psi]$,
$\bm V_{2,12}^{w,3}[\psi] =\bm V_{2,12}^{3}[w\psi]$,
$\bm  V_{2,21}^{w,1}[\bm \phi] =\bm  V_{2,21}^{1}[w \bm \phi]$,
$\bm  V_{2,21}^{w,2}[\bm \phi] =\bm  V_{2,21}^{2}[w^{-1}\bm \phi]$,
$\bm V_{2,21}^{w,3}[\bm \phi] =\bm V_{2,21}^{3}[w^{-1}\bm \phi]$,
$\bm V_{2,22}^{w,1} [\psi]= \bm V_{2,22}^{1} [w \psi]$,
$\bm V_{2,22}^{w,2} [\psi] =\bm V_{2,22}^{2} [w^{-1}\psi]$.
The integral operators 
\begin{align*}
    \bm V_{2,12}^{1}[\psi]=& \frac{\gamma k_p^{2} \bm\nu_{\bm x}}{k_1^{2}-k_2^{2}} \int_{\Gamma}  \pa_{ \bm\nu_{\bm y}}[\gamma_{k_1}(\bm x,\bm y)-\gamma_{k_2}(\bm x,\bm y)]\psi(\bm y)d s_{\bm y}, \\ 
    \bm V_{2,12}^{2}[\psi]=& \frac{2\mu \gamma}{(k_1^{2}-k_2^{2})(\lambda+2\mu)} \int_{\Gamma} \nabla_{\bm y} (\gamma_{k_1}(\bm x,\bm y)-\gamma_{k_2}(\bm x,\bm y))\psi(\bm y) d s_{\bm y} ,\\
    \bm V_{2,12}^{3}[\psi]=&-\frac{2\mu \gamma}{(k_1^{2}-k_2^{2})(\lambda+2\mu)}\int_{\Gamma} [k_1^{2}\gamma_{k_1}(\bm x,\bm y)-k_2^{2}\gamma_{k_2}(\bm x,\bm y)]\mathbb{A} \bm \nu_{\bm y} \psi(\bm y)d s_{\bm y},\\
   \bm  V_{2,21}^{1}[\bm \phi]=&\frac{i\omega\eta k_p^{2}}{k_1^{2}-k_2^{2}}\int_{\Gamma} \pa_{\nu_{\bm x}}[\gamma_{k_1}(\bm x,\bm y)-\gamma_{k_2}(\bm x,\bm y)]\bm \nu_{\bm y}^{\top}\bm \phi(\bm y)d s_{\bm y}, \\
   \bm  V_{2,21}^{2}[\bm \phi]=&\frac{2i\mu \omega\eta}{(k_1^{2}-k_2^{2})(\lambda+2\mu)}  \int_{\Gamma} \nabla_{\bm x}^{\top}[\gamma_{k_1}(\bm x,\bm y)-\gamma_{k_2}(\bm x,\bm y)]\bm \phi(\bm y) d s_{\bm y},\\
    \bm V_{2,21}^{3}[\bm \phi]=&-\frac{2i\mu \omega\eta}{(k_1^{2}-k_2^{2})(\lambda+2\mu)} \int_{\Gamma} [k_1^{2}\gamma_{k_1}(\bm x,\bm y)-k_2^{2}\gamma_{k_2}(\bm x,\bm y)]\bm \nu_{\bm x}^{\top}\mathbb{A} \bm \phi (\bm y) d s_{\bm y}, \\
    \bm V_{2,22}^{1} [\psi] =&-\frac{1}{k_1^{2}-k_2^{2}}\int_{\Gamma}[k_1^{2}(k_p^{2}-k_1^{2})\gamma_{k_1}(\bm x,\bm y)-k_2^{2}(k_p^{2}-k_2^{2})\gamma_{k_2}(\bm x,\bm y)]\bm\nu_{\bm x}^{\top}\bm\nu_{\bm y} \psi(\bm y) d s_{\bm y}, \\
    \bm V_{2,22}^{2} [\psi]=&-\frac{1}{k_1^{2}-k_2^{2}}\int_{\Gamma}[(k_p^{2}-k_1^{2})\gamma_{k_1}(\bm x,\bm y)-(k_p^{2}-k_2^{2})\gamma_{k_2}(\bm x,\bm y)] \psi(\bm y) d s_{\bm y},
\end{align*}  
are at most weakly-singular.
Similarly, the operator $\bm V_{3,12}^w,\bm V_{3,21}^w,\bm V_{3,22}^w$ can be re-expressed as
\begin{align*}
    \bm V_{3,12}^w=&\bm V_{3,12}^{w,1} +\bm V_{3,12} ^{w,2} D_s^w,\\
    \bm V_{3,21}^w=&\bm V_{3,21}^{w,1} + D_s \bm V_{3,21}^{w,2}, \\
   \bm  V_{3,22}^w=&\bm V_{2,22}^w,
\end{align*}
where  $\bm V_{3,12}^{w,1} [\psi]= \bm V_{3,12}^{1} [w\psi]$,
$\bm V_{3,12}^{w,2} [\psi] =\bm V_{3,12}^{2} [w^{-1}\psi]$,
$\bm V_{3,21}^{w,1} [\bm \phi] =\bm V_{3,21}^{1} [w^{-1}\bm \phi]$,
$\bm V_{3,21}^{w,2} [\bm \phi]= \bm V_{3,21}^{2} [w^{-1}\bm \phi]$,
and the integral operators
\begin{align*}
    \bm V_{3,12}^{1} [\psi] =&\frac{\gamma}{(k_1^{2}-k_2^{2})(\lambda+2\mu)} \int_{\Gamma} (-k_1^{2}\gamma_{k_1}(\bm x,\bm y)+k_2^{2}\gamma_{k_2}(\bm x,\bm y))\psi(y)\bm\nu_{\bm y}  d s_{\bm y}, \\
    \bm V_{3,12}^{2} [\psi] =&-\frac{\gamma}{(k_1^{2}-k_2^{2})(\lambda+2\mu)}\int_{\Gamma} \mathbb{A} \nabla_{\bm y} (\gamma_{k_1}(\bm x,\bm y)-\gamma_{k_2}(\bm x,\bm y))\psi(\bm y) d s_{\bm y}, \\
    \bm V_{3,21}^{1} [\bm \phi] =&\frac{i\omega \eta}{(k_1^{2}-k_2^{2})(\lambda+2\mu)} \int_{\Gamma} (-k_1^{2}\gamma_{k_1}(\bm x,\bm y)+k_2^{2}\gamma_{k_2}(\bm x,\bm y))\bm\nu_{\bm x}^{\top}  \bm \phi(\bm y)d s_{\bm y}, \\
    \bm V_{3,21}^{2} [\bm \phi] =&-\frac{i\omega\eta}{(k_1^{2}-k_2^{2})(\lambda+2\mu)}  \int_{\Gamma} \nabla_{\bm x}^{\top}(\gamma_{k_1}(\bm x,\bm y)-\gamma_{k_2}(\bm x,\bm y))\mathbb{A}  \bm \phi(\bm y)d s_{\bm y},
\end{align*}
are all at most weakly-singular.

Finally, we present the expressions of $\mathbb{E}^1(t,\tau)$ and $\mathbb{E}^2(t,\tau)$ in (\ref{anal}). 
 Using the series expansions of Bessel functions~\cite{CK13} , when $t\neq \tau$ we obtain
\begin{align*}
 \mathbb{E}_{11}^1(t,\tau)=&-\frac{1}{2\pi\mu}J_0 (k_s r)\mathbb{I}_2+\frac{1}{2\pi\rho \omega^2 r}[k_s J_1 (k_s r)-\frac{k_p^2-k_2^2}{k_1^2-k_2^2}k_1 J_1(k_1 r)+\frac{k_p^2-k_1^2}{k_1^2-k_2^2}k_2 J_1(k_2 r)]\\
 &-\frac{(\bm x(t)-\bm x(\tau))(\bm x(t)-\bm x(\tau))^{\top}}{2\pi\rho\omega^2 r^2}[k_s^2 J_2(k_s r)-\frac{k_p^2-k_2^2}{k_1^2-k_2^2}k_1^2 J_2 (k_1 r)+\frac{k_p^2-k_1^2}{k_1^2-k_2^2}k_2^2 J_2 (k_2 r)],  \\
 \bm E_{12}^1(t,\tau)=&\frac{i\omega \eta}{(k_1^2-k_2^2)(\lambda+2\mu)}\frac{1}{2\pi r}[k_1 J_1(k_1 r) -k_2 J_1(k_2 r)](\bm x(t)-\bm x(\tau)), \\
  \bm E_{21}^1(t,\tau)=&-\frac{ \gamma}{(k_1^2-k_2^2)(\lambda+2\mu)}\frac{1}{2\pi r}[k_1 J_1(k_1 r) -k_2 J_1(k_2 r)](\bm x(t)-\bm x(\tau)), \\
  E_{22}^1 (t,\tau)=&\frac{1}{2\pi}[\frac{1}{k_1^2-k_2^2}(k_p^2-k_1^2)J_0 (k_1 r)-(k_p^2-k_2^2) J_0 (k_2 r)],
 \end{align*}
with $r=|x(t)-x(\tau)|$, $J_n$ denoting Bessel functions of order $n$ and
\ben \mathbb{E}^2 (t,\tau)=\mathbb{E} (\bm x(t),\bm x(\tau))-\mathbb{E}^1 (t,\tau) \log|t-\tau|. \enn
When $t=\tau$, we have
\begin{align*}
\mathbb{E}_{11}^1 (t,t)=&-\frac{1}{2\pi\mu}\mathbb{I}_2+\frac{k_s^2-k_p^2}{4\pi\rho\omega^2}\mathbb{I}_2, \\
 \mathbb{E}_{11}^2 (t,t)=&\frac{i}{4\mu} [1+\frac{2i}{\pi}(\log \frac{k_s |\bm x'(t)|}{2}+C_e)]\mathbb{I}_2-\frac{i}{4\rho\omega^2}\frac{k_s^2-k_p^2}{2}[1+\frac{2i}{\pi}(\log|\bm x'(t)|+C_e)-\frac{i}{\pi}]\mathbb{I}_2 \\
 +&\frac{k_s^2-k_p^2}{4\pi\rho\omega^2}\frac{\bm x'(t)^{\top}\bm x'(t)}{|\bm x'(t)|^2}+\frac{1}{4\pi\rho\omega^2}(k_s^2\, \log\frac{k_s}{2}-\frac{k_p^2-k_2^2}{k_1^2-k_2^2}k_1^2\, \log\frac{k_1}{2}+\frac{k_p^2-k_1^2}{k_1^2-k_2^2}k_2^2 \,\log\frac{k_2}{2})\mathbb{I}_2,\\
 \bm E_{12}^1 (t,t)=&\bm E_{12}^2(t,t)=0, \\
 \bm E_{21}^1(t,t)=&\bm E_{21}^2(t,t)=0, \\
 E_{22}^1(t,t)=&-\frac{1}{2\pi}, \\
 E_{22}^2(t,t)=&\frac{i}{4}-\frac{C_e}{2\pi}+\frac{1}{2\pi}(\frac{k_p^2-k_1^2}{k_1^2-k_2^2}\log \frac{k_1 \vert x'(t) \vert}{2}-\frac{k_p^2-k_2^2}{k_1^2-k_2^2}\log \frac{k_2 \vert x'(t) \vert}{2}),
 \end{align*}
where $C_e$ is the Euler constant.

  \end{document}